\documentclass[11pt,a4paper]{article}
\usepackage[a4paper,left=2.5cm,right=2.5cm,top=2.5cm,bottom=2.5cm]{geometry}
\usepackage{amsfonts}
\usepackage{amsmath}
\usepackage{amsthm}
\usepackage[colorlinks=True,
            linkcolor=blue,
            anchorcolor=green,
            citecolor=blue
            ]{hyperref} 
\usepackage[noabbrev,capitalise,nameinlink]{cleveref}
\usepackage{float}
\usepackage{graphicx}
\usepackage{epstopdf}
\usepackage{multirow}
\usepackage{caption,subcaption}

\usepackage{algorithm}
\usepackage{algorithmic}

\newtheorem{theorem}{Theorem}[section]

\newtheorem{lemma}{Lemma}[section]
\newtheorem{proposition}{Proposition}[section]

\newtheorem{definition}{Definition}[section]
\newtheorem{remark}{Remark}[section]

\newtheorem{example}{Example}[section]

\def\R{\mathbb{R}}

\def\bx{{\bf x}}
\def\by{{\bf y}}
\def\bz{{\bf z}}
\def\bu{{\bf u}}
\def\bd{{\bf d}}
\def\bv{{\bf v}}

\def\be{{\boldsymbol\eta }}
\def\bt{{\boldsymbol\theta }}
 
\graphicspath{ {./images/} }

\author{
  Jun Sun, Lingchen Kong, and Shenglong Zhou\\
  Department of Applied Mathematics, Beijing Jiaotong University, CN\\  
  Department of EEE, Imperial College London, UK \\
  
}

\title{\vspace{-1.25cm}
Gradient Projection Newton Algorithm for Sparse Collaborative Learning Using Synthetic and Real Datasets of Applications
\vspace{-0.25cm}}

\date{}

\begin{document}

\maketitle
 
\vspace{-1.3cm}

\begin{abstract}  
\noindent \textbf{Abstract:} Exploring the relationship among  multiple sets of data from one same group enables practitioners to make better decisions in medical science and engineering.   In this paper, we propose a sparse collaborative learning (SCL) model, an optimization with double-sparsity constraints, to process the problem with two sets of data and a shared response variable. It is capable of dealing with the classification problems or the regression problems dependent on the discreteness of the response variable as well as exploring the relationship between two datasets simultaneously. To solve SCL, we first present some necessary and sufficient optimality conditions and then design a gradient projection Newton algorithm which has proven to converge to a unique locally optimal solution  globally with at least a quadratic convergence rate. Finally, the reported numerical experiments illustrate the efficiency of the proposed method.
\vspace{0.3cm}
 
\end{abstract}

\numberwithin{equation}{section}
\section{Introduction}\label{sec:intr}

There are many scenarios where datasets from the same group can be collected  {from various sources. Therefore, they differ  but interact} \cite{Lenz,VISS,Wangs,Zhangh}. For example, a researcher studying cancer outcomes may collect gene expression data and copy number data from a group of patients. The traditional approaches to do predictions are either merging two datasets or using two datasets separately. Both ways ignore the fact that they are from different sources with different meanings (e.g., gene expression and copy number). As stated in  \cite{Thom}, exploring the relationship between sources allows for extracting informative biomarkers and improving clinical outcome predictions. Motivated by such practical applications,  {in this paper,  we propose the following sparse collaborative learning (SCL) problem}:
\begin{eqnarray}\label{005}
	\begin{aligned}
		\min_{\bt_{1},\bt_{2}}~& \frac{1}{n} \left[ a \cdot \ell(\bt_1;X,\by) + b  \cdot \ell(\bt_2;Z,\by)+\frac{c}{2}\|X\bt_{1}-Z\bt_{2}\|^{2}\right] =:f(\bt_{1},\bt_{2}) \\
		{\rm s.t.}&\quad \|\bt_{1}\|_{0}\leq s_{1},~ \|\bt_{2}\|_{0}\leq s_{2},\\
	\end{aligned}
\end{eqnarray}
 {where $\ell(\cdot)$ is a general loss function,} $X\in \mathbb{R}^{n\times p_{1}}, Z\in  \mathbb{R}^{n\times p_{2}}$ are two  datasets from two different sources and $\by\in  \mathbb{R}^{n}$ is the shared response, $n$ is the sample/subject size, and $p_{1}, p_{2}$ represent the feature/variable sizes of two datasets. Here, $\|\bt\|_0$ is the zero norm of $\bt$, counting the number of its nonzero elements, $s_{1}\ll p_{1},~s_{2}\ll p_{2}$ are two integers representing the prior information on the upper bounds of the signal sparsity, 
$a, b$ and $c$ are positive parameters, and $\|\cdot\|$ represents the Euclidean norm.  SCL models have been applied into many real-world applications, such as face recognition by using a mixture of synthetic and real images with dynamic weight \cite{Feng}, medical diagnosis including schizophrenia, Alzheimer’s disease, or various neurocognitive phenotypes by using genetic and imaging data \cite{GRO,Zill,Hu}.

Two typical examples of $\ell$ will be investigated in this paper. When $\ell$ is the linear regression loss,
$$ \ell_{lin}(\bt;X,\by):= \frac{1}{2}\sum_{i=1}^{n} ( y_{i}- \langle\bx_{i}, \bt\rangle)^{2}, $$
SCL is called sparse collaborative regression (SCoRe \cite{GRO}) usually working for the continuous response  $\by$. Here, $\langle \bx, \bz\rangle$ is the inner product of two vectors $\bx$ and $\bz$ and $\bx_{i}$ is a column vector corresponding the $i$-th row of $X$. SCoRe  is a combination of linear regression and canonical correlation analysis (CCA).  The former makes predictions via employing two different types of datasets and the latter explores the relationship between them. Examples of employing $\ell_{lin}$ include CoRe  \cite{Brefeld}, multi-task CoRe \cite{Zill} and the models studied in \cite{Feng,Hu}. 

We note that the aforementioned models based on $\ell_{lin}$ aimed to process the continuous response $\by$. However, various real-world applications involve discrete responses, in particular for those in classification problems including the severity of the disease, whether or not to die and to name a few. Under such circumstances, linear regression-based models are unlikely to provide accurate predictions and hence it is necessary to consider the logistic regression loss defined by,
$$ \ell_{log}(\bt;X,\by):= \sum_{i=1}^{n} \Big(\log \left(1+\exp \langle\bx_{i}, \bt\rangle\right)-y_{i}\langle \bx_{i}, \bt\rangle \Big). $$
SCL with such a loss is called the sparse logistic collaborative regression (SLCoRe), which can be used to deal with datasets with discrete response $\by$. SLCoRe is a combination of logistic regression and CCA, aiming at classifying the samples in each of the two datasets while exploring the relationship between them. It is well-known that discrete responses are frequently involved in classification problems, while most of the existing classification methods including support vector machines \cite{Liux,Zhangx} and logistic regression  \cite{Liu,Plan,Wang,wang2021extended} only target one dataset. Very little work makes predictions for multiple sets of data and explores the relationship among them at the same time. 

However, to accurately characterize the sparsity, it is suggested to impose the sparsity constraints directly instead of using the approximations/regularizations. For example, Beck and Eldar \cite{Beck} thoroughly studied a general sparsity-constrained optimization model and developed the famous iterative hard thresholding algorithm, in the meanwhile, Bahmani et al. \cite{Bahm} and
Plan et al. \cite{Plan} investigated the logistic regression model with sparsity constraints. After which there is a vast body of work on developing optimization algorithms and understanding the properties of various sparse estimators for the sparsity constrained optimization \cite{Wang, Panl, Pan,   ZXQ21, zhou2022}.  We emphasize that all those work aimed at addressing applications with single datasets rather than multiple datasets. 

It this paper, we study two typical examples of SCL: SLCoRe with $\ell=\ell_{log}$ and SCoRe with $\ell=\ell_{lin}$. All  results to be established are based on these two models. The main contributions of the paper are summarized as follows:
\begin{itemize}
	\item[I)] We propose a unified framework, SCL, for the problems with discrete or continuous response variables and two different sets of data. It can classify or predict the data in each dataset, and explore the relationship between the two datasets. New model \eqref{005} exploits the sparsity constraints directly, which enables to select a sufficiently small portion of informative features in each dataset provided that
	$s_{1}$ and $s_{2}$ are small enough.
	\item[II)]  We investigate the first-order necessary and sufficient optimality conditions (see Theorem \ref{theo01-suff-necc} and Theorem \ref{theo01}) for SCL as well as  the existence and the uniqueness of its solution (see Theorem \ref{theo04}). One of the optimality conditions is associated with the $\alpha$-stationary point seen  Definition \ref{def3.1} that allows for algorithmic design conveniently.
	
	\item[III)]  We develop a gradient projection Newton algorithm (GPNA) that combines the gradient projection motivated by the $\alpha$-stationary point and the Newton step to accelerate the convergence. We prove that GPNA not only converges to a unique local minimizer of  problem \eqref{005} globally (see Theorem \ref{global-convergence}) but also has a quadratic convergence rate for SLCoRe and termination within finite steps for SCoRe (see Theorem \ref{theo07}) under a mild assumption. These nice convergence properties indicate that our proposed algorithm should behave excellently in terms of high accuracy and speed, which is testified by its outstanding numerical performance.
\end{itemize}

We note that SCL problem (\ref{005}) has a close link to the multi-model problem where multiple models based on the learned data distributions are used to make predictions \cite{Goldfeld1973,Richard1984,Xiao2018,Zhaolei2021}. In contrast, SCL  focuses on two groups of data not only for the prediction but also for exploring their inter-group relationships. To this end,  if two groups of data in the dataset are known, then SCL with $c=0$ (namely, no inter-group relationships are investigated) in problem (\ref{005}) can be deemed as a special case of the multi-model problem.

To end this section, we present the organization of this paper. The next section describes the notation that will be employed through this paper and displays some properties of the objective function of problem \eqref{005}. In  section \ref{sec:opt}, we establish the first-order necessary and sufficient optimality conditions as well as the existence and the uniqueness of the solutions to  problem \eqref{005}.  The algorithm GPNA and its convergence properties are provided in section \ref{sec:gra}. Numerical experiments on synthetic and real data are reported in section \ref{sec:num}, and some concluding remarks are given in the section \ref{sec:con}.

\section{Preliminaries}\label{sec:pre}

 {Before giving the main results, we define some notations that will be employed throughout the paper}.  {Let $[p]:=\{1,2,\ldots,p\},[n]:=\{1,2,\ldots,n\}$}.  We denote  sparse set $\Sigma_{s}^p$ in $\mathbb{R}^p$ by 
$$\Sigma_{s}^p:=\{\bt\in \mathbb{R}^{p}:\|\bt \|_{0}\leq s\},$$ where $s \ll p$ is an integer.  For a vector $\bt$, denote its neighborhood with a radius $\delta$ by $ N(\bt,\delta):=\{\bu\in\mathbb{R}^p: \|\bt-\bu\|< \delta\}$, and its support set by  $\Gamma(\bt):=\{i\in [p]:\theta_{i}\neq 0\}$. The complement set of $\Gamma$ is written as $\overline{\Gamma}$. For a given set $T$, its spanned subspace of $\mathbb{R}^p$ is denoted by $\mathbb{R}^{p}_{T}:=\{\bt\in\mathbb{R}^p: \Gamma(\bt)\subseteq T\}$. Let  $\bt_{\Gamma}$ be the subvector of $\bt$ indexed on $\Gamma$. We merge two vectors $\bt_1$ and $\bt_2$ as a single column vector via $(\bt_1;\bt_2):=(\bt_1^\top\bt_2^\top)^\top$. Finally, for a matrix $A\in\mathbb{R}^{n\times p}$, let $\lambda_{\max}(A)$ and  $\lambda_{\min}(A)$ present its largest and smallest eigenvalue, respectively, and $A_{TJ}$ denotes the sub-matrix containing rows indexed by ${T}$ and columns indexed by $J$. In particular, $A_{T:}:=A_{T[p]}$ and $A_{:J}:=A_{[n]J}$.

To characterize the projection of $\bt$ onto $\Sigma_{s}^p$, we denote $\bt^\downarrow_{i}$ the $i$th largest element in magnitude of $\bt$. Based on this,  projection  $\Pi_{\Sigma_{s}^p}(\bt)$ that is given by
\begin{eqnarray}
	\Pi_{\Sigma_{s}^p}(\bt):=\underset{\bu\in \Sigma_{s}^p}{\rm argmin}~\|\bt-\bu\|\nonumber
\end{eqnarray}
can be derived as follows: If $\bt^\downarrow_{s}=0$ or $\bt^\downarrow_{s}>\bt^\downarrow_{s+1}$, then it is unique, i.e., 
$$(\Pi_{\Sigma_{s}^p}(\bt))_{i}=
\begin{cases}
	\bt_{i},&~~\mid\bt_{i}\mid \geq \bt^{\downarrow}_{s},\\
	0,&~~\mid\bt_{i}\mid < \bt^{\downarrow}_{s}.
\end{cases}$$
If there are more than one equal to $\bt^\downarrow_{s}$, we can choose any one of them and let the rest be 0. If  $\bt^\downarrow_{s}=\bt^\downarrow_{s+1}\neq 0$, then 
$$(\Pi_{\Sigma_{s}^p}(\bt))_{i}=
\begin{cases}
	\bt_{i},&~~\mid\bt_{i}\mid> \bt^\downarrow_{s},\\
	\bt_{i} ~~{\rm or}~~ 0,&~~\mid\bt_{i}\mid= \bt^\downarrow_{s},\\
	0,&~~\mid\bt_{i}\mid< \bt^\downarrow_{s}.
\end{cases}$$
For example, for $\bt=\{2,4,3,-3,1\}$ and $\Sigma_{2}^5=\{\bu\in \mathbb{R}^{5}:\|\bu \|_{0}\leq 2\}$, we have $\Pi_{\Sigma_{2}^5}(\bt)=(0,4,3,0,0)^\top$ or $(0,4,0,-3,0)^\top$.

Below are some concepts that will be used in this paper.
\begin{definition}
	[$s$-regularity \cite{Beck}] A matrix $A\in \mathbb{R}^{n\times p}$ is called $s$-regular if its any $s$ columns are linearly independent.
\end{definition}

\begin{definition} [Strong smoothness \cite{Jala}]
	If  function $f$ is continuously differentiable, then for any $\bt,\bd\in \mathbb{R}^{p}$, we say that   function $f$ is strongly smooth on $\mathbb{R}^{p}$ with a parameter $L_f>0$ if it holds that
	$$f(\bt+\bd) \leq f(\bt)+\langle\nabla f(\bt), \bd\rangle+({L_f}/{2})\|\bd\|^{2}.$$
\end{definition}
\begin{definition} [Restricted strong convexity \cite{Bahm,ZXQ21,Agarwal,Shalev-Shwartz}]  
	If  function $f$ is  twice continuously differentiable, then for any $\bt, \bd\in \Sigma_r^p$ satisfying $\bt+\bd\in \Sigma_r^p$, we say that   function $f$ is $r$-restricted strongly convex on $\Sigma_r^p$ with a parameter $l_f>0$ if it holds that
	$$
	f(\bt+\bd) \geq f(\bt)+\langle\nabla f(\bt), \bd\rangle+ ({l_f}/{2})\|\bd\|^{2} \quad \text { or }~~~~ \langle \bd, \nabla^{2} f(\bt) \bd\rangle \geq  ({l_f}/{2})\|\bd\|^{2}.
	$$
	If these conditions hold for $l_f=0$, then $f$ is called $r$-restricted  convex on $\Sigma_r^p$.
\end{definition} 
We now give some properties of $f$ in problem (\ref{005}), including the strong smoothness and restricted strong convexity  as well as the Lipschitz continuity of its gradient and Hessian matrix.
\begin{proposition}\label{Pro1} Let $\bt:=(\bt_{1};\bt_{2})$ and $\ell=\ell_{log}$. Objective function $f$  in problem (\ref{005}) has the following properties.
	\begin{itemize}
		\item[1)] It is convex, twice continuously differentiable and 
		strongly smooth with parameter $L_f$ given by
		$$L_f:=\lambda_{\max}\left(\frac{1}{n}\left[\begin{array}{cc}
			( {a }/{4}+c)X^{\top}X & - cX^{\top}Z \\ - cZ^{\top}X & ( {b}/{4}+c)Z^{\top}Z \end{array}\right]\right),
		$$
		which indicates that $\nabla f$ is Lipschitz continuous with   parameter $L_f$ for any $\bt$ and $\bt'$,
		\begin{eqnarray}\label{Lip-grad}\|\nabla f(\bt)-\nabla f(\bt')\| \leq L_f \| \bt  -\bt'\|.\nonumber\end{eqnarray}
		\item[2)] Its Hessian matrix $\nabla^{2}f(\bt)$ takes the form of
		$$\nabla^{2}f(\bt)=\frac{1}{n}\begin{bmatrix} 
			X^{\top}(a  D_{1}+cI)X & - cX^{\top}Z \\ 
			-cZ^{\top}X &  Z^{\top}(b D_{2}+cI)   Z  \end{bmatrix},$$
		where $I$ is the identity matrix, $D_{1} $ and $D_2$ are two diagonal matrices with $$(D_1)_{ii}=\frac{\exp \langle \bx_{i},\bt_{1} \rangle }{(1+\exp \langle \bx_{i},\bt_{1} \rangle)^{2}},~i\in[n],$$
		$$(D_2)_{ii}=\frac{\exp \langle \bz_{i},\bt_{2} \rangle}{(1+\exp \langle \bz_{i},\bt_{2} \rangle)^{2}},~~i\in[n].$$
		Moreover, $\nabla^{2}f(\cdot)$ is Lipschitz continuous with   constant $C_f$, namely, 
		\begin{eqnarray}\label{Lip-hassian}\|\nabla^2 f(\bt)-\nabla^2 f(\bt')\| \leq C_f \| \bt  -\bt'\|,\end{eqnarray}
		for any $\bt$ and $\bt'$, where 
		$$C_f:=
		\frac{3\sqrt{2}}{n}\max\left\{ a  \max_{i\in[n]}\|\bx_i\|_1\lambda_{\max}(X^\top X),  b\max_{i\in[n]}\|\bz_i\|_1 \lambda_{\max}(Z^\top Z)\right\}.$$

		\item[3)] If matrix $[X~Z]$ is $(s_{1}+s_{2})$-regular, then it is $(s_{1}+s_{2})$-restricted strongly convex on $\Sigma^{p_{1}+p_{1}}_{s_1+s_2}$  with a positive parameter $l_f$ given by 
		\begin{eqnarray}\label{s-regular-1}l_f:=\min_{|T|\leq s_{1}+s_{2}}\lambda_{\min}\left(\frac{c}{n}\left[\begin{array}{cc}
				X^{\top}X & - X^{\top}Z \\ 
				- Z^{\top}X &  Z^{\top}Z \end{array}\right]_{TT}\right).\end{eqnarray}
	\end{itemize}
\end{proposition}

\begin{proof} 
	1) It is easy to see that $f$ is convex and twice continuously differentiable. Since $t/(1+t)^{2}\leq 1/4$ for any $t\geq0$, it follows $\lambda_{\max}(\nabla^{2}f(\bt)) \leq L_f$ for any $\bt\in \mathbb{R}^{p_1+p_2}$. This can show that the gradient of $f$ is Lipschitz continuous with   parameter $L_f$ immediately. 
	
	2) It follows from \cite[Lemma A.3]{Wang} that $\nabla^2\ell(\bt_1;X)$ and both $\nabla^2\ell(\bt_2;Z)$  are Lipschitz continuous with   constants $$C_1:=(3/n) \max_{i\in[n]}\|\bx_i\|_1\lambda_{\max}(X^\top X),~~~~C_2:=(3/n) \max_{i\in[n]}\|\bz_i\|_1\lambda_{\max}(Z^\top Z).$$ Then we have
	\begin{eqnarray*} 
		\arraycolsep=1.4pt\def\arraystretch{1.5}
		\begin{array}{llll}
			\|\nabla^2 f(\bt)-\nabla^2 f(\bt')\|  &=& \|a  \nabla^2\ell(\bt_1;X) + b \nabla^2\ell(\bt_2;Z)- a  \nabla^2\ell(\bt_1';X) - b \nabla^2\ell(\bt_2';Z)\| \\
			&\leq&  a  C_1  \|  \bt_1 -\bt_1'\| + b C_2 \| \bt_2-\bt_2'\|\\
			&\leq&  \max\{ a  C_1, b C_2\} (\|  \bt_1 -\bt_1'\| +  \| \bt_2-\bt_2'\|)\\
			&\leq&  \sqrt{2}\max\{ a  C_1, b C_2\}  \|  \bt  -\bt '\|.
		\end{array}
	\end{eqnarray*}
	3) If  matrix $[X~Z]$ is $(s_{1}+s_{2})$-regular, then so is  matrix $[X~-Z]$. Note that
	$$\nabla^{2}f(\bt)=\frac{1}{n}\begin{bmatrix} 
		a  X^{\top}  D_{1} X & 0 \\ 
		0 &  b Z^{\top}  D_{2}   Z  \end{bmatrix} +
	\frac{c}{n}\begin{bmatrix} 
		X^{\top} X & - X^{\top}Z \\ 
		-Z^{\top}X &  Z^{\top}  Z  \end{bmatrix} =: A+ B.$$
	Clearly, both $A$ and $B$ are  positive semi-definite. Moreover, $B=(c/n)[X~-Z]^\top[X~-Z]$. According to the  $(s_{1}+s_{2})$-regular of the matrix $[X~Z]$, we can get $B_{TT}$ is  positive definite. Therefore, for any $\bd :=(\bd_1;\bd_2)\neq 0$ with $\|\bd_1\|_0\leq s_1$ and $\|\bd_2\|_0\leq s_2$, we have
	\begin{eqnarray*}
		\langle \bd,\nabla^{2} f(\bt)\bd\rangle   = \langle \bd,  (A+B) \bd\rangle  
		\geq  \langle \bd,  B \bd\rangle \geq l_f \|\bd\|^2>0. 
	\end{eqnarray*}
	This displays that the $(s_{1}+s_{2})$-restricted strong  convexity of $f(\bt)$ on $\Sigma_{s_1+s_2}^{p_1+p_2}$.
 {	The proof is complete.\qed}
	
\end{proof}

We note that the classical logistic regression which has been shown to be only strictly convex instead of being restricted strongly convex even though the assumption of the regularity of the sample matrix is imposed. However, the objective function of SLCoRe can be restricted strongly convex if the sample matrix is regular. In addition, if we only have one dataset, SLCoRe will degenerate into the classical sparse logistic regression. At this point, see the example in \cite{Wang}, similar results can be obtained. Similarly,  for the objective function of SCoRe, we easily obtain the following results.

\begin{proposition}\label{Pro2}   Let $\bt:=(\bt_{1};\bt_{2})$ and $\ell=\ell_{lin}$. Objective function $f$  in (\ref{005})  is convex, twice continuously differentiable and has Hessian matrix $\nabla^{2}f(\bt)$ in the form of
	$$\nabla^{2}f(\bt)=\frac{1}{n}\begin{bmatrix} 
		(a +c)X^{\top}X & - cX^{\top}Z \\ 
		-cZ^{\top}X &  (b+c)Z^{\top} Z  \end{bmatrix}=:Q.$$
	Moreover, it is	strongly smooth with   parameter $L_f:=\lambda_{\max}\left( Q\right)$ and thus $\nabla f$ is Lipschitz continuous with  parameter $L_f$.  If    $[X~Z]$ is $(s_{1}+s_{2})$-regular, then it is $(s_{1}+s_{2})$-restricted strongly convex on   $ \Sigma^{p_{1}+p_{1}}_{s_{1}+s_{2}}$ with a positive parameter $l_f>0$ given by 
	\begin{eqnarray}\label{s-regular}l_f:=\min_{\mid T\mid\leq s_{1}+s_{2}}\lambda_{\min}\left(Q_{TT}\right).
	\end{eqnarray}
\end{proposition}
It is worth mentioning that the main  theorems in the sequel are established based on the assumption of $s$-regularity. So, 
to end this section, we would like to see which types of matrices $[X~Z]$ could satisfies $s$-regularity. To proceed with that, we introduce the famous Restricted Isometry Property (RIP, \cite{candes2005decoding}). A matrix $\Phi\in\R^{n\times p}$ is said to satisfy $s$-order RIP, if there exists a constant $\delta_s\in[0,1)$ such that
\begin{eqnarray*}
(1-\delta_s)\|\bt\|^2 \leq \|\Phi\bt\|^2 \leq (1+\delta_s)\|\bt\|^2 
\end{eqnarray*}
for all vectors $\bt\in\Sigma_s^p$. This definition is equivalent to $$(1-\delta_s) \leq \lambda_{\min}(\Phi_{:T}^\top \Phi_{:T}) \leq \lambda_{\max}(\Phi_{:T}^\top \Phi_{:T}) \leq  (1+\delta_s), ~~\forall~ |T|\leq s.$$
Therefore, matrices satisfying $s$-order RIP  must satisfy $s$-regularity. On the other hand, it has proven in \cite{candes2006near,baraniuk2008simple} that random Gaussian matrix, random binary matrix, and Fourier matrix satisfy $s$-order RIP with a high probability when $s$ is small enough. Hence, these matrices also satisfy $s$-regularity.

\section{Optimality Conditions}\label{sec:opt}

This section establishes the optimality conditions of SCL being useful for the algorithmic development,  before which, for notational convenience, we define 
\begin{eqnarray}\label{notation-1}
	\arraycolsep=1.4pt\def\arraystretch{1.5}
	\begin{array}{rll}
		\bt&:=&(\bt_1;\bt_2),\\
		\nabla_{i} f (\bt)&:=&\nabla_{\bt_{i}} f (\bt ),~~~~i=1,2,\\
		\Sigma_i&:=& \Sigma_{s_i}^{p_i},~\hspace{12mm} i=1,2,\\
		\Sigma&:=& \{\bt\in\mathbb{R}^{p_1+p_2}:\bt_1\in \Sigma_1,\bt_2\in \Sigma_2\},\\
		s&:=&s_1+s_2.
	\end{array}
\end{eqnarray}
Similar rules are also applied for $\bt_{1}^{\ast}$ and $\bt_{2}^{\ast}$. Based on these notation, we now establish the first-order necessary and sufficient  optimality conditions for  problem (\ref{005}).
\begin{theorem}\label{theo01-suff-necc} Let $\bt^{*}$ be a point that satisfies
	\begin{eqnarray}\label{l1-KKT}
		\arraycolsep=1.4pt\def\arraystretch{1.5}
		\begin{array}{rll}
			(\nabla_{j}f(\bt^{*}))_{i} =0,&~i\in \Gamma(\bt_{j}^{*}), &~~{\rm if}~~ \|\bt_{j}^{*}\|_{0}=s_{j},\\
			\nabla_{j}f(\bt^{*})=0,&&~~{\rm if}~~\|\bt_{j}^{*}\|_{0}<s_{j},
		\end{array}
	\end{eqnarray}
	for $j=1,2$. Then $\bt^{*}$ is a local minimizer of  (\ref{005}) if and only if it satisfies (\ref{l1-KKT}).  
\end{theorem} 
\begin{proof}Necessity. Based on \cite[Theorem 6.12]{RW09}, a local minimizer $\bt^{*}$ of the  problem (\ref{005}) must satisfy that $- \nabla f(\bt^{*}) \in  \mathcal{N}_{\Sigma}(\bt^{*}) 
	= \mathcal{N}_{\Sigma_1}(\bt^{*}_1) \times \mathcal{N}_{\Sigma_2}(\bt^{*}_2),$ 
	where $\mathcal{N}_{\Sigma}(\bt^{*})$ is the normal cone of $\Sigma$ at $\bt^{*}$ and the equality is by \cite[Theorem 6.41]{RW09}. Then the explicit expression (see \cite[Table 1]{Panl}) of  normal cone $\mathcal{N}_{\Sigma_j}(\bt^{*}_j)$ enable us to derive \eqref{l1-KKT} immediately.
	
	Sufficiency. Let $\bt^{*}$ satisfy \eqref{l1-KKT}. The convexity of $f$ leads to
	\begin{eqnarray}\label{convex-f}
		f(\bt)\geq f(\bt^{*})+\langle \nabla_{1} f(\bt^{*}), \bt_{1}-\bt_{1}^{*}\rangle+\langle\nabla_{2} f(\bt^{*}), \bt_{2}-\bt_{2}^{*}\rangle.\nonumber
	\end{eqnarray}
	If there is a $\delta>0$ such that, for any $\bt\in \Sigma\cap N(\bt^*,\delta)$,
	\begin{eqnarray}\label{fact-1}
		\langle \nabla_{1}f(\bt^{*}),\bt_{1}-\bt_{1}^{*}\rangle=\langle \nabla_{2}f(\bt^{*}),\bt_{2}-\bt_{2}^{*}\rangle=0, 
	\end{eqnarray} 
	then the conclusion can be made immediately. Therefore, we next to show \eqref{fact-1}. 
	In fact, by \eqref{l1-KKT}, we note that $\nabla_{j}f(\bt^{*})=0$ if   $\|\bt_{j}^{*}\|_{0}<s_{j}$, which indicates it suffices to consider the worst case of $\|\bt_{j}^{*}\|_{0}=s_{j},j=1,2$. Under such a case, we define 
	$$ \delta   := \min_{j=1,2} \min\limits_{i \in \Gamma(\bt_{j}^{*})} \mid(\bt_{j}^{*})_i\mid.$$
	Then for any  $\bt\in N(\bt^{*},\delta)\cap \Sigma$, we have  
	\begin{eqnarray*}
	\forall~i \in \Gamma(\bt_{j}^{*}),~~	\mid(\bt_{j})_i\mid&=&\mid(\bt_{j})_i^*- (\bt_{j})_i^*+(\bt_{j})_i\mid\\
		&\geq& \mid(\bt_{j})_i^*\mid-\mid (\bt_{j})_i^*-(\bt_{j})_i\mid\\
		&\geq& \mid(\bt_{j})_i^*\mid-\| \bt_{j}^*-\bt_{j}\|\\
		&>&  \mid(\bt_{j})_i^*\mid-\delta\\
		&\geq&0.\end{eqnarray*}
The above relationship means that $i \in \Gamma(\bt_{j}^{*})$ (i.e. $(\bt_{j}^{*})_i\neq0$ ) implies $|(\bt_{j})_i|>0$ (i.e. $(\bt_{j})_i\neq0$), which leads to $\Gamma(\bt_{j}^{*})\subseteq \Gamma(\bt_{j})$. This by $\|\bt_{j}\|_{0}\leq s_{j}=\|\bt_{j}^{*}\|_0=|\Gamma(\bt_{j}^{*})|$ allows us to yield that
	$$\Gamma(\bt_{j}^{*})= \Gamma(\bt_{j}), j=1,2,~\forall~\bt\in N(\bt^{*},\delta)\cap  \Sigma.$$
	Using the above fact and \eqref{l1-KKT} derives that
	\begin{eqnarray*} 
		\langle \nabla_{j}f(\bt^{*}),\bt_{j}-\bt_{j}^{*}\rangle=\langle (\nabla_{j}f(\bt^{*}))_{\Gamma(\bt_{j}^{*})},(\bt_{j}-\bt_{j}^{*})_{\Gamma(\bt_{j}^{*})}\rangle=0.
	\end{eqnarray*} 
	 {	The proof is complete.\qed}
\end{proof}
Based on Theorem \ref{theo01-suff-necc}, however, the necessary and sufficient  optimality conditions \eqref{l1-KKT} mean that there is no useful information for the case $i\notin \Gamma(\bt_{j}^{*})$ when $\|\bt_{j}^{*}\|_0=s_j$. So, we introduce the concept of the $\alpha$-stationary point of  (\ref{005}).
\begin{definition} \label{def3.1}
	We say that $\bt^{\ast}$ is an $\alpha$-stationary point of  problem (\ref{005}) if there exists an $\alpha>0$ such that
	$$\bt_{1}^{*} \in \Pi_{\Sigma_1} (\bt_{1}^{*}-\alpha \nabla_{1} f (\bt^{*}) ),\quad \bt_{2}^{*} \in \Pi_{\Sigma_2}\left(\bt_{2}^{*}-\alpha \nabla_{2} f\left(\bt^{*}\right)\right).$$
\end{definition}

If there is only one variable, the definition of the $\alpha$-stationary points is the same as that in  \cite{Beck, Panl} which allows us to derive its explicit expression as follows.
\begin{lemma}\label{lamma-3.1}
	For a given $\alpha>0$,  the point $ \bt^{*} $ is an $\alpha$-stationary point of  problem (\ref{005}) if and only if for $j=1,2$, it satisfies  
	\begin{eqnarray}\label{l1}
		\arraycolsep=1.4pt\def\arraystretch{1.5}
		\begin{array}{cl}
			\alpha (\nabla_{j}f(\bt^{*}))_{i}
			\left\{\begin{array}{ll}
				=0,&~i\in \Gamma(\bt_{j}^{*}),\\
				\leq (\bt^*)^\downarrow_{s}, &~i\in\overline{\Gamma}(\bt_{j}^{*}),\\
			\end{array} \right.&~~{\rm if}~~ \|\bt_{j}^{*}\|_{0}=s_{j},\\
			\nabla_{j}f(\bt^{*})=0,&~~{\rm if}~~\|\bt_{j}^{*}\|_{0}<s_{j}.
		\end{array}
	\end{eqnarray}
\end{lemma}
Comparing conditions \eqref{l1-KKT} and \eqref{l1}, the latter provides more information for the case of $i\in\overline{\Gamma}(\bt_{j}^{*})$. It can be clearly seen that the latter is a stronger condition and  suffices to the former.  

The following result reveals the relationships among the $\alpha$-stationary point and the global/local minimizers of problem (\ref{005}).
\begin{theorem}\label{theo01} Let $\bt^{*}$ be an $\alpha$-stationary point of problem (\ref{005}), then it is a local minimizer. Furthermore, if $\|\bt_{1}^{*}\|_{0}<s_{1},\|\bt_{2}^{*}\|_{0}<s_{2}$, then it is also a global minimizer. Conversely, if $\bt^{*}$ is a  global minimizer of problem (\ref{005}), then it is an $\alpha$-stationary point with $0<\alpha< 1/L_f$.
\end{theorem}
\begin{proof}Since condition \eqref{l1} imply \eqref{l1-KKT} and a point satisfying \eqref{l1-KKT} is a local minimizer by Theorem \ref{theo01-suff-necc}, an $\alpha$-stationary point of  (\ref{005}) is a local minimizer.
	
	Conversely, suppose that a  global minimizer  $\bt^{*}$ of problem (\ref{005}) is not an  $\alpha$-stationary point with $0<\alpha< 1/L_f$, that is, there exists $\be^*_{1}\neq \bt_{1}^{*}$ or $\be^*_{2}\neq \bt_{2}^{*}$ such that
	$$\be^*_{1} \in \Pi_{\Sigma_1}\left(\bt_{1}^{*}-\alpha \nabla_{1} f\left(\bt^{*}\right)\right)~~\text{or}~~\be^*_{2} \in \Pi_{\Sigma_2}\left(\bt_{2}^{*}-\alpha \nabla_{2} f\left(\bt^{*}\right)\right).$$
	Without loss of any generality, we have  both  of the above  conditions. Then
	$$\|\be^*_{j}-\bt_{j}^{*}+\alpha \nabla_{j} f(\bt^{*})\|^{2}\leq \|\bt_{j}^{*}-\bt_{j}^{*}+\alpha \nabla_{j} f(\bt^{*})\|^{2}, ~~j=1,2,$$
	by the definition of   projection $\Pi(\cdot)$, which implies
	$$ \langle\be^*_{j}-\bt_{j}^{*}, \nabla_{j} f(\bt^{*})\rangle\leq-(1/{2\alpha})\|\be^*_{j}-\bt_{j}^{*}\|^{2}, ~~i=1,2.$$
	Using this condition and the strong smoothness of $f$   results in
	\begin{eqnarray*}
		\arraycolsep=1.4pt\def\arraystretch{1.5}
		\begin{array}{cl}
			f(\be^*)&\leq~f(\bt^{*})+ \langle   \nabla f(\bt^{*}), \be^*-\bt^{*}\rangle+({L_f }/{2})\|\be^*-\bt^{*}\|^{2}\\
			&\leq f(\bt^{*})+\left({L_f }/{2}- {1}/({2\alpha})\right) \|\be^* -\bt ^{*}\|^{2}  <f(\bt^{*}),
		\end{array}
	\end{eqnarray*}
	where the last inequality is from $0<\alpha< 1/L_f$. The above condition contradicts with the optimality of $\bt^{*}$. So $\bt^{*}$  is  an  $\alpha$-stationary point with $0<\alpha< 1/L_f$.  {The proof is complete.\qed}
\end{proof}
To end this section, we would like to see the existence and uniqueness of  solutions to problem (\ref{005}), which is revealed by the following theorem.
\begin{theorem}\label{theo04} If matrix $[X~Z]$ is $s$-regular, then the global minimizer of problem (\ref{005}) exists, and the local minimizers are finitely many and each of them is unique.
\end{theorem}
\begin{proof} Based on our notation $\mathbb{R}^{p_j}_{T_j}=\{\bt_j\in\mathbb{R}^{p_j}: \Gamma(\bt_j)\subseteq T_j\}$ , we note that $\bt_j\in \mathbb{R}^{p_j}_{T_j}$ implies  $|\Gamma(\bt_j)|\leq |T_j|, j=1,2$. Therefore,  original problem (\ref{005})  is equivalent to
  \begin{eqnarray*}
 		\begin{array}{cl}
 		\min\limits_{\bt_{1},\bt_2} &  f(\bt_{1},\bt_2) \\
 		{\rm s.t.}&  \bt_{1} \in \mathbb{R}^{p_1}_{T_1},~~ \forall ~|T_1| = s_{1}, \\
 	& \bt_{2} \in 	\mathbb{R}^{p_2}_{T_2},~~ \forall ~|T_2| = s_{2}.  
  	\end{array} 
 \end{eqnarray*}
This problem is clearly equivalent to  
	\begin{eqnarray}\label{005-1}
		\begin{aligned}
			\min_{\mid T_1\mid =s_{1},\mid T_2\mid =s_{2}}\left\{\min_{\bt_{1},\bt_{2}}~ f(\bt_{1},\bt_{2}),~~
			{\rm s.t.}~ \bt_{1}\in \mathbb{R}^{p_{1}}_{T_1},~\bt_{2}\in \mathbb{R}^{p_{2}}_{T_2}\right\}.
		\end{aligned}
	\end{eqnarray}
	If matrix $[X~Z]$ is $s$-regular, then $f$ is $s$-restricted strongly convex on $\Sigma^{p_1+p_2}_{s} =\Sigma^{p_1+p_2}_{s_1+s_2} $ by Proposition \ref{Pro1} or Proposition \ref{Pro2}, and hence it is $s$-restricted strongly convex on
	$\mathbb{R}^{p_{1}}_{T_1}\times \mathbb{R}^{p_{2}}_{T_2}$ due to $\mathbb{R}^{p_{1}}_{T_1}\times \mathbb{R}^{p_{2}}_{T_2} \subseteq \Sigma^{p_1+p_2}_{s_1+s_2}$. 
	
 It follows from \cite[Lemma 6]{nesterov2009primal} that the inner program admits a unique global minimizer denoted by  $(\bt_{1}^*(T_1),\bt_{2}^*(T_2))$. Note that $T_1\subseteq[p_1]$ and $T_2\subseteq[p_2]$. Thus there are finitely many  $T_1$ and $T_2$ such that $\mid T_1\mid=s_{1}$ and  $\mid T_2\mid=s_{2}$, and so are the inner programs. This indicates that  $(\bt_{1}^*(T_1),\bt_{2}^*(T_2))$  is finitely many. To derive the global minimizer of \eqref{005-1}, we only pick one $(\bt_{1}^*(T_1),\bt_{2}^*(T_2))$ that makes the objective function value of \eqref{005-1} minimal. Therefore, the global minimizers exist.
	
	We next show that any local minimizer $\bt^*$ is unique. To proceed with that, denote   $\delta:=\min\{\delta_1,\delta_2\}$ where 
	\begin{eqnarray*}
		\arraycolsep=1.4pt\def\arraystretch{2}
		\begin{array}{cll}
			\delta_j&:=& 
			\begin{cases}
				+\infty, & \bt_{j}^* =0,\\
				\min_{i\in \Gamma(\bt_{j}^*)}   \mid(\bt_{j}^*)_i\mid,  & \bt_{j}^* \neq0,\
			\end{cases}~~~~j=1,2. 
		\end{array}
	\end{eqnarray*}
	Clearly,  $\delta_1, \delta_2>0$ and hence $\delta>0$. 
	Then, similar reasoning allows us to derive \eqref{fact-1} for any $\bt\in \Sigma \cap  N(\bt^*,\delta)$. This and  $f$ being $s$-restricted strongly convex lead to
	\begin{eqnarray}\label{convex-f-unique}
		f(\bt)\geq f(\bt^{*}) + (l_f/2) \|\bt -\bt^{*}\|^2.\nonumber
	\end{eqnarray}
	The above condition indicates $\bt^{*}$ is the unique global minimizer of  problem $\min \{f(\bt):\bt\in \Sigma \cap  N(\bt^*,\delta)\}$, namely, $\bt^{*}$ is the unique local minimizer of problem \eqref{005}.  {The proof is complete.\qed}
\end{proof}

\section{Gradient Projection Newton   Algorithm} \label{sec:gra}

In this section, we propose the gradient projection Newton algorithm (GPNA) for  problem  (\ref{005}). Again, for notational simplicity, we define some notations
\begin{eqnarray}\label{notation-k}
	\arraycolsep=1.4pt\def\arraystretch{1.5}
	\begin{array}{lllllllll}
		\bu^k&:=&(\bu_1^k;\bu_2^k), &~~~~ \bt^k&:=&(\bt_1^k;\bt_2^k), \\
		\Gamma_k&:=&\Gamma(\bu^k),&~~~~H^k&:=&\nabla^{2} f(\bu^k),\\
		\bt^k(\alpha)&:=& (\bt_1^k(\alpha);\bt_2^k(\alpha)),&
		~~~~\bt_j^k(\alpha)&\in&  \Pi_{\Sigma_j}(\bt_{j}^{k}-\alpha \nabla_j f(\bt^k)),~ j=1,2.\nonumber
	\end{array}
\end{eqnarray}
Based on the notation in \eqref{notation-1}, we actually have
\begin{eqnarray}\label{notation-k-1}
	\bt^k(\alpha) \in  \Pi_{\Sigma}(\bt^{k}-\alpha \nabla f(\bt^k)).
\end{eqnarray}

The algorithmic framework of GPNA  summarized in Algorithm \ref{algorithm 1} consists of two major components. The first one is based on the two projected gradient steps, which enforces two variables to satisfy the sparsity constraints. The second part adopts a Newton step to speed up the convergence. However, the Newton step is only performed when one of the following conditions is satisfied, 
 {\begin{eqnarray} \label{Newton-switch-on}
	\arraycolsep=1.4pt\def\arraystretch{1.5}
	\begin{array}{llll}
		\text{Condition 1)}~ &  \Gamma(\bt_1^k)=\Gamma(\bu^k_1),&\Gamma(\bt_2^k)=\Gamma(\bu^k_2),\\
		\text{Condition 2)}~&    \|\nabla_1 f(\bu^k)\|<\epsilon,~&\Gamma(\bt_2^k)=\Gamma(\bu^k_2),\\
		\text{Condition 3)}~&   \|\nabla_2 f(\bu^k)\|<\epsilon,~&\Gamma(\bt_1^k)=\Gamma(\bu^k_1),\\
		\text{Condition 4)}~&   \|\nabla_1 f(\bu^k)\|<\epsilon ,& \|\nabla_2 f(\bu^k)\|<\epsilon,
	\end{array}
\end{eqnarray}}
where $\epsilon>0$ is a given tolerance. 

\begin{algorithm}[!th]
	\caption{GPNA: Gradient Projection Newton Algorithm \label{algorithm 1}}
	\begin{algorithmic}[1]
		\REQUIRE Initialize $\bt^{0}$. Let $ 0< \sigma, 0< \epsilon, 0<\alpha_{0} \leq1, 0<\gamma<1, 0 <\varepsilon < {\tt tol}_0$ and set $k\Leftarrow0$.
		\WHILE{${\tt tol}_k > \varepsilon$}
		\STATE \underline{Gradient projection:} Find the smallest integer $q_k = 0, 1, \cdots$ such that
		\STATE \parbox{.92\textwidth}{\begin{eqnarray}\label{armijio-descent-property} f(\bt^k(\alpha_{0}\gamma^{q_{k}}))\leq f(\bt^{k})-({\sigma}/{2})\|\bt^k(\alpha_{0}\gamma^{q_{k}})-\bt^{k}\|^{2}.\nonumber\end{eqnarray}}
		\STATE Set $\alpha_k= \alpha_{0}\gamma^{q_{k}}$,  $ \bu^k=\bt^k(\alpha_k)$ and $\bt^{k+1}=\bu^{k}$.
		\IF{one of the conditions in (\ref{Newton-switch-on}) is satisfied}
		\STATE\underline{Newton step:}
		If the following equations are solvable
		\STATE \parbox{.88\textwidth}{\begin{eqnarray}\label{Newton-descent-property}
				H^k_{\Gamma_{k}\Gamma_{k}} (\bv^{k}_{\Gamma_k}-\bu^{k}_{\Gamma_k})=-(\nabla f(\bu^{k}))_{\Gamma_k}, ~~\bv^{k}_{\overline\Gamma_k}=0,
		\end{eqnarray}} 
		\STATE and the solution $\bv^k$ satisfies
		\STATE \parbox{.88\textwidth}{\begin{eqnarray}\label{Newton-descent-property-1}
				f(\bv^k)\leq f(\bu^k) - ({\sigma}/{2})\|\bv^k-\bu^k\|^2,
		\end{eqnarray}}
		\STATE then set $\bt^{k+1}=\bv^{k}$.			
		\ENDIF
		\STATE	Compute ${\tt tol}_k:=\|(\nabla f(\bt^{k+1}))_{\Gamma_k}\|$ and set $k:= k+1$.
		\ENDWHILE
		\STATE	Output the solution $\bt^k.$
	\end{algorithmic}
\end{algorithm}

\begin{remark}We have some comments on the halting condition and computational complexity for GPNA in Algorithm \ref{algorithm 1}.
	\begin{itemize}
		\item One can discern that if $\bt^{k+1}=\bu^k$,  
		then $ \bt^{k+1}_{\overline\Gamma_k}=0$.  If $\bt^{k+1}=\bv^k$, then the updating rule \eqref{Newton-descent-property} for $\bv^k$ indicates that
		\begin{eqnarray} \label{Gv-Gu}
			\Gamma(\bv^k)\subseteq\Gamma_k =\Gamma(\bu^k),
		\end{eqnarray}
		which also implies $ \bt^{k+1}_{\overline\Gamma_k}=0$. Now suppose   ${\tt tol}_k =0$, i.e., $(\nabla f(\bt^{k+1}))_{\Gamma_k}=0$.  Then $\bt^{k+1}$ satisfies (\ref{l1-KKT}) and thus is a local minimizer of problem (\ref{005}). Therefore, it makes sense to terminate the algorithm when ${\tt tol}_k < \varepsilon$.

		\item We note that the calculations of $\Pi_{\Sigma_1}, \Pi_{\Sigma_2}$ and gradient   $\nabla f$  dominate the computation for the gradient projection step. And these three terms are easy to calculate and their total computational complexity is about $O(n(p_1+p_2))$. For the Newton step, if matrix $[X~Z]$ is $s$-regular,  then the inverse of $H^{k}_{\Gamma_{k}\Gamma_{k}}$ exists due to $\mid\Gamma_{k}\mid\leq s$, which means that every Newton step is well defined. Moreover, the worst-case computational complexity of deriving $\bv^{k}$ is about $O(s^3+ns^2)$. Overall, the entire computational complexity of the $k$th iteration of Algorithm \ref{algorithm 1} is $O(s^3+ns^2 + q_k n(p_1+p_2))$. We prove that $\alpha_k$ is bounded by upper and lower bounds. If we know the strong smooth parameter $L_{f}$ of the objective function $f$, then $q_k$ may be taken as 1 or a small positive integer.
		
	\end{itemize}
	
\end{remark}

\subsection{Global convergence}
Before establishing the main convergence results, we define a constant $\underline{\alpha}$  by
\begin{eqnarray}
	\label{lower-bd-alphak}\underline{\alpha}:=\min\left\{1,~ {\gamma}{(\sigma+L_f)^{-1}}\right\},\nonumber
\end{eqnarray}
which is a positive scalar. We first need the following lemma.
\begin{lemma} \label{theo05}
	Let $\{\bt^{k}\}$ be the sequence generated by GPNA. The following statements are true.
	\begin{itemize}
		\item[1)] For any $0<\alpha\leq  1/( \sigma+L_f)$, it holds that
		\begin{eqnarray}
			\label{zk-alpha-zk}f(\bt^k(\alpha))\leq f(\bt^k) - (\sigma/2)\|\bt^k(\alpha)-\bt^k\|^2,
		\end{eqnarray}
		and thus $\inf_{k\geq 0}\{\alpha_k\}\geq \underline{\alpha}>0$.		
		\item[2)] $\{f(\bt^k)\}$ is a non-increasing sequence and $$\lim\limits_{k\rightarrow \infty}\|\bu^{k}-\bt^{k}\|=\lim\limits_{k\rightarrow \infty}\|\bt^{k+1}-\bt^{k}\|=0.$$ 
		\item[3)] Any accumulating point of  sequence  $\{\bt^k\}$ is an $\alpha$-stationary point with $0<\alpha\leq \underline{\alpha}$ of problem  (\ref{005}). 
	\end{itemize}
\end{lemma}
\begin{proof}1) It follows from \eqref{notation-k-1} that $\bt^k(\alpha)  \in \Pi_{\Sigma }(\bt^k -\alpha \nabla  f(\bt^k))$ and thus
	$$\| \bt ^k(\alpha)-( \bt ^k - \alpha  \nabla   f(\bt^k))\|^2 \leq \|\bt ^k -( \bt ^k - \alpha  \nabla   f(\bt^k))\|^2,$$
	which results in
	\begin{eqnarray} \label{eta-point-not-3}
		2\alpha\langle   \nabla   f(\bt^k),  \bt^k(\alpha)-  \bt^k\rangle \leq -  \| \bt ^k(\alpha)-  \bt^k\|^2.\nonumber
	\end{eqnarray}
	This and the strong smoothness of $f$ with  constant $L_f$ derive that
	\begin{eqnarray*}
		\arraycolsep=1.4pt\def\arraystretch{1.5}
		\begin{array}{lll}
			f(\bt^k(\alpha)) &\leq&    f(\bt^k)+
			\langle  \nabla  f(\bt^k),\bt^k(\alpha)-\bt^k\rangle
			+ (L_{f}/2)\|\bt^k(\alpha)-\bt^k\|^2\\
			& {\leq}&   f(\bt^k)- ( {1}/{(2\alpha)} -(L_{f}/2) )\|\bt^k(\alpha)-\bt^k\|^2\\
			&\leq&  f(\bt^k) -  (\sigma/2)\|\bt^k(\alpha)-\bt^k\|^2,
		\end{array}
	\end{eqnarray*}
	where the last inequality is from $0<\alpha\leq  1/(\sigma+L_{f})$.  {Invoking} the Armijo-type step size rule,  {one has} $\alpha_k\geq \gamma/(\sigma+L_{f})$,  which by  $\alpha_k\leq 1$ proves the desired assertion. 
	
	2) By \eqref{zk-alpha-zk} and $\bu^k=\bt^k(\alpha_k)$, we have
	\begin{eqnarray}\label{u-z-ell}
		f(\bu^k)&\leq&  f(\bt^k) -  (\sigma/2)\|\bu^k-\bt^k\|^2.
	\end{eqnarray}
	By the framework of Algorithm \ref{algorithm 1}, if $\bt^{k+1}=\bu^k$, then the above condition implies,
	\begin{eqnarray}
		\label{fact-u-z}		 f(\bt^{k+1})&\leq&
		f(\bt^k) - (\sigma/2)\|\bt^{k+1}-\bt^k\|^2.\nonumber
	\end{eqnarray}
	If $\bt^{k+1}=\bv^k$, then  we obtain
	\begin{eqnarray}\label{fact-u-z-1}
		\arraycolsep=1.4pt\def\arraystretch{1.5}
		\begin{array}{lll}
			f(\bt^{k+1})= f(\bv^k)&\leq&f(\bu^k) - (\sigma/2)\|\bt^{k+1} -\bu^k\|^2 \\
			&\leq& f(\bt^k) - (\sigma/2)\|\bu^k-\bt^k\|^2- (\sigma/2)\|\bt^{k+1}-\bu^k\|^2\\
			&\leq& f(\bt^k) -(\sigma/4)\|\bt^{k+1}-\bt^k\|^2, \nonumber
		\end{array}
	\end{eqnarray}
	where the second and last inequalities used \eqref{u-z-ell} and a fact $\|{\bf a}+{\bf b}\|^2\leq 2\|{\bf a}\|^2+2\|{\bf b}\|^2$  {for all vectors ${\bf a}$ and ${\bf b}$}. Both cases lead to	 \begin{eqnarray}\label{decrese-pro}
		\arraycolsep=1.4pt\def\arraystretch{1.5}
		\begin{array}{lll}
			f(\bt^{k+1})&\leq& f(\bt^k) - (\sigma/4)\|\bt^{k+1}-\bt^k\|^2,\\  
			f(\bt^{k+1})&\leq& f(\bt^k) -  (\sigma/2)\|\bu^k-\bt^k\|^2. 
		\end{array}
	\end{eqnarray}
	Therefore, $\{f(\bt^k)\}$ is  non-increasing, which by \eqref{decrese-pro} and $f \geq 0$ yields
	\begin{eqnarray*} &&\sum_{k\geq 0} \max\{(\sigma/4)\|\bt^{k+1}-\bt^k\|^2, (\sigma/2)\|\bu^k-\bt^k\|^2\}\\
		&\leq& \sum_{k\geq 0} \left[ f(\bt^k) - f(\bt^{k+1})\right] = f(\bt^0) - \lim_{k\rightarrow\infty} f(\bt^{k+1})\leq f(\bt^0).	
	\end{eqnarray*}
	The above condition suffices to $\lim_{k\rightarrow\infty}\|\bt^{k+1}-\bt^k\|=\lim_{k\rightarrow\infty}\|\bu^k-\bt^k\|=0.$ 
	
	3) Let $\bt^*$ be any accumulating point of $\{\bt^k\}$.  {Then there exists a} subset $M$ of $\{0,1,2,\ldots\}$ such that $\lim_{ k (\in M)\rightarrow\infty} \bt^k = \bt^*.$  {This further implies $\lim_{ k (\in M)\rightarrow\infty} \bu^k = \bt^*$ by applying 2). 
		In addition, as stated in 1), we have $\{\alpha_k\}\subseteq [ \underline{\alpha }, 1]$,
		which indicates that one can find a subsequence $K$ of $M$ and a scalar $\alpha _*\in [ \underline{\alpha }, 1]$ such that $\{\alpha_k:  k \in K\}\rightarrow \alpha _*$. Overall, we have}
	\begin{eqnarray}
		\label{a-a-under} \lim_{ k (\in K)\rightarrow\infty}\bt^k =\lim_{ k (\in K)\rightarrow\infty}\bu^k = \bt^*, ~~~~\lim_{ k (\in K)\rightarrow\infty} \alpha_k =\alpha _* \in [\underline{\alpha },1].\end{eqnarray}
	Let $\be^k :=\bt^k -\alpha_k \nabla  f(\bt^k)$. The framework of Algorithm \ref{algorithm 1} implies
	\begin{eqnarray}\label{u-l-P}
		\bu^k  \in \Pi_{\Sigma}( \be^k ),~~~~ \lim_{ k (\in K)\rightarrow\infty} \be^k =\bt^* -\alpha _*  \nabla f(\bt^*)=:\be^*.\end{eqnarray}
	The first condition means $\bu^k  \in \Sigma$ for any $ k \geq1$. Note that $\Sigma$ is closed and  $\bt^*$ is the accumulating point of $\{\bu^k\}$ by \eqref{a-a-under}. Therefore, $\bt^*\in \Sigma$, which results in
	\begin{eqnarray}
		\label{z-*-g}
		\min_{\bt\in \Sigma}\|\bt -\be^*\| \leq \|\bt^*-{\be}^*\|.\end{eqnarray}
	If `$<$' holds in the above condition, then there is an $ \varepsilon_0>0$ such that
	\begin{eqnarray*}
		\|\bt^*-\be^*\|-\varepsilon_0 &=&  \min_{\bt\in \Sigma}\|\bt -\be^*\| \\
		&\geq&\min_{\bt\in \Sigma} (\|\bt - \be^k \|-\|\be^k- \be^*\|)\\
		&=&\|\bu^k - \be^k \|-\|\be^k -\be^*\|,
	\end{eqnarray*}
	where the last equality is from \eqref{u-l-P}. Taking the limit of both sides of the above condition along $ k (\in K)\rightarrow\infty$ yields $\|\bt^*- \be^*\|-\varepsilon_0 \geq \|\bt^* - \be^*\|$ by \eqref{a-a-under} and \eqref{u-l-P}, a contradiction with  $ \varepsilon_0>0$. Therefore, we must have the equality holds in \eqref{z-*-g}, showing that
	\begin{eqnarray*} \bt^* \in  \Pi_ \Sigma(\be^*)= \Pi_\Sigma\left( \bt^* - \alpha _*   \nabla f(\bt^*)\right). \end{eqnarray*}
	The above relation means the conditions in  \eqref{l1} hold  for $\alpha =\alpha _*$, then these conditions must hold for any $0<\alpha \leq \underline{\alpha }$ due to $\underline{\alpha }\leq \alpha _*$ from \eqref{a-a-under}, namely,
	\begin{eqnarray*} \bt^* \in   \Pi_\Sigma\left( \bt^* - \alpha  \nabla f(\bt^*)\right), \end{eqnarray*}
	displaying that $\bt^*$ is an $\alpha $-stationary point of problem \eqref{005}, as desired.  {The proof is complete.\qed}
\end{proof}
The above lemma allows us to conclude that the whole sequence converges.
\begin{theorem}\label{global-convergence}  Let  $\{\bt^ k \}$ be the sequence generated by GPNA. Then the whole sequence converges to  a unique local minimizer of  (\ref{005}) if $[X~Z]$ is $s$-regular.
\end{theorem}
\begin{proof}  As shown in Lemma \ref{theo05}, $\{\bt^k\}\subseteq \{\bt: f(\bt)\leq f(\bt^0),\bt\in \Sigma \}$ is a bounded set due to $s$-restricted strong convexity of $f$ from the $s$-regularity of $[X~Z]$. Therefore, one can find  a subsequence of $\{\bt^ k \}$ that converges to $\alpha$-stationary point $\bt^*$ with $0<\alpha\leq \underline{\alpha}$ of problem \eqref{005}. 
	Recall that an $\alpha $-stationary point $\bt^*$ is also a local minimizer by Theorem \ref{theo01}, which by Theorem \ref{theo04} indicates that $\bt^*$ is unique if $[X~Z]$ is $s$-regular. In other words, $\bt^*$ is an isolated local minimizer of problem (\ref{005}).  Finally, it follows from $\bt^*$ being isolated, \cite[Lemma 4.10]{more1983computing} and $\lim_{ k \rightarrow\infty}\|\bt^{ k +1}-\bt^ k \|=0$ by Lemma \ref{theo05} that
	the whole sequence converges to the unique local minimizer, $\bt^*$.  {The proof is complete.\qed} 
\end{proof}
\subsection{Convergence rate}
This part aims to establish the convergence rate of GPNA when the sequence falls into a local area of its limiting point. Before the main result, we claim the following facts. 
\begin{lemma}\label{quadratic-lemma}  Suppose $[X~Z]$ is $s$-regular. Let  $\{\bt^ k \}$ be the sequence generated by GPNA and $\bt^*$ be its limit.  The following results hold for sufficiently large $k$.
	\begin{itemize}
		\item[1)] The support set of $\bt^*$ can be identified by
		\begin{eqnarray}\label{support-identify}
			\arraycolsep=1.4pt\def\arraystretch{1.5}
			\Gamma(\bt_j^*)\left\{\begin{array}{lll}
				\subseteq (\Gamma(\bt_j^k) \cap\Gamma(\bu_j^k)),&~~\text{if}~~ & \|\bt_j^*\|_0<s_j,\\
				\equiv \Gamma(\bt_j^k)\equiv \Gamma(\bu_j^k),&~~\text{if}~~& \|\bt_j^*\|_0=s_j,
			\end{array}\right.~~~~j=1,2.
		\end{eqnarray}
		\item[2)] The Newton step is always admitted if we set $\sigma\in(0,l_f/2)$.
	\end{itemize}
\end{lemma}
\begin{proof}1) If $\|\bt_j^*\|_0=s_j$, then by $\bt_j^k\rightarrow \bt_j^*, \bu_j^k\rightarrow \bt_j^*$ and $\|\bt_j^k\|_0\leq s_j,\|\bu_j^k\|_0\leq s_j$, we must have $\Gamma(\bt_j^*) 
	\equiv \Gamma(\bt_j^k)\equiv \Gamma(\bu_j^k)$ for sufficiently large $k$. If $\|\bt_j^*\|_0<s_j$, similar reasoning allows for deriving $\Gamma(\bt_j^*) 
	\subseteq   \Gamma(\bt_j^k) $ and $ \Gamma(\bt_j^*) 
	\subseteq \Gamma(\bu_j^k)$.
	
	2) By Theorem \ref{global-convergence}, the limiting point, $\bt^*$, is a local minimizer of  problem (\ref{005}). Therefore, it satisfies \eqref{l1-KKT} from Theorem \ref{theo01-suff-necc}. We first conclude that for sufficiently large $k$, one of the four conditions in \eqref{Newton-switch-on} must be satisfied. In fact, there are four cases for $\bt^*$ and each case can imply one   condition in \eqref{Newton-switch-on}  as follows: 
	 {\begin{eqnarray} 
		\arraycolsep=1.4pt\def\arraystretch{1.5}
		\begin{array}{lll}
			\text{Case 1)}&~\|\bt_1^*\|_0=s_1,~\|\bt_2^*\|_0=s_2~~~~\Longrightarrow~~~~\text{Condition 1)},\\
			\text{Case 2)}&~\|\bt_1^*\|_0<s_1,~\|\bt_2^*\|_0=s_2~~~~\Longrightarrow~~~~\text{Condition 2)},\\
			\text{Case 3)}&~\|\bt_1^*\|_0=s_1,~\|\bt_2^*\|_0<s_2~~~~\Longrightarrow~~~~\text{Condition 3)},\\
			\text{Case 4)}&~\|\bt_1^*\|_0<s_1,~\|\bt_2^*\|_0<s_2~~~~\Longrightarrow~~~~\text{Condition 4)}.\\\nonumber
		\end{array}
	\end{eqnarray}}
	We now show them one by one. The Lipschitz continuity of $\nabla f$ indicates that
	\begin{eqnarray} \label{Lip-g}
		\arraycolsep=1.4pt\def\arraystretch{1.5}
		\begin{array}{llll}
			&&\max\{\|\nabla_j f(\bu^k)- \nabla_j f(\bt^*)\|,\|(\nabla f(\bu^{k}))_{\Gamma_k}-(\nabla f(\bt^*))_{\Gamma_k}\}\\
			&\leq &\|\nabla f(\bu^k)- \nabla f(\bt^*)\| \leq L_f\| \bu^k - \bt^* \|.
		\end{array}
	\end{eqnarray}
	The relation of Case 1) $\Rightarrow$ Condition 1) can be derived by \eqref{support-identify} immediately. For Case 2), we have $\Gamma(\bt_2^k)\equiv \Gamma(\bu_2^k)$ by \eqref{support-identify} and 
	\begin{eqnarray} 
		\arraycolsep=1.4pt\def\arraystretch{1.5}
		\begin{array}{llll}
			\|\nabla_1 f(\bu^k)\| &=& \|\nabla_1 f(\bu^k)- \nabla_1 f(\bt^*)\|&~~(\text{by \eqref{l1-KKT}})\\
			&\leq &L_f\| \bu^k - \bt^* \| & ~~(\text{by \eqref{Lip-g}})\\
			&\leq& \epsilon.& ~~(\text{by $ \bu^k \rightarrow \bt^*$}) \nonumber
		\end{array}
	\end{eqnarray}
	Therefore, Case 2) $\Rightarrow$ Condition 2). Similarly, we can show the last two relations. 
	
	Next, since $[X~Z]$ is $s$-regular, 
	$H^k_{\Gamma_{k}\Gamma_{k}}$ is non-singular, which means that the equations \eqref{Newton-descent-property} are solvable. Finally, we show the inequality \eqref{Newton-descent-property-1} is true when $\sigma\in(0,l_f/2)$. In fact, the conditions \eqref{support-identify} and \eqref{l1-KKT} enable to derive  
	\begin{eqnarray}\label{b*-gamma}
		(\nabla f(\bt^*))_{\Gamma_k}=0,
	\end{eqnarray}	
	for sufficiently large $k$. 
	Then it follows from 	 \eqref{Newton-descent-property} that 
	\begin{eqnarray*} 
		\arraycolsep=1.4pt\def\arraystretch{1.5}
		\begin{array}{llll}
			\|\bv^k-\bu^k\| &=& \|\bv^{k}_{\Gamma_k}-\bu^{k}_{\Gamma_k}\|&~~(\text{by \eqref{Gv-Gu}})\\ 
			&=& \|(H^k_{\Gamma_{k}\Gamma_{k}})^{-1}(\nabla f(\bu^{k}))_{\Gamma_k}\|&~~(\text{by   \eqref{Newton-descent-property}})\\
			&\leq&(1/l_f)\|(\nabla f(\bu^{k}))_{\Gamma_k}\|&~~(\text{by \eqref{s-regular-1} or \eqref{s-regular}})\\
			&=&(1/l_f)\|(\nabla f(\bu^{k}))_{\Gamma_k}-(\nabla f(\bt^*))_{\Gamma_k}\|&~~(\text{by   \eqref{b*-gamma}})\\
			&\leq&(L_f/l_f) \|\bu^{k}-\bt^*\|\rightarrow0.& ~~(\text{by  \eqref{Lip-g}})  
		\end{array}
	\end{eqnarray*}
	The above condition indicates that $\|\bv^k-\bu^k\|\rightarrow 0$, resulting in
	\begin{eqnarray}\label{vk-uk}
		o(\|\bv^k-\bu^k\|^2) \leq (l_f/{4}) \|\bv^k-\bu^k \|^2, 
	\end{eqnarray}	
	for sufficiently large $k$.	 Now,  we have the following chain of inequalities,
	\allowdisplaybreaks
	\begin{eqnarray*}
		\arraycolsep=1.4pt\def\arraystretch{1.5}
		\begin{array}{llll}
			2f(\bv^k)-2f(\bu^k)&=& 2\langle\nabla f(\bu^k),\bv^k-\bu^k \rangle+2o(\|\bv^k-\bu^k\|^2)\\
			&+& \langle \nabla^2f(\bu^k)(\bv^k-\bu^k),\bv^k-\bu^k \rangle~~~~(\text{by Taylor expansion})\\
			&=&  2\langle (\nabla f(\bu^k))_{\Gamma_k},(\bv^k-\bu^k)_{\Gamma_k} \rangle+2o(\|\bv^k-\bu^k\|^2)\\
			&+&  \langle H^k_{\Gamma_{k}\Gamma_{k}}(\bv^k-\bu^k)_{\Gamma_k},(\bv^k-\bu^k)_{\Gamma_k} \rangle ~~~~ (\text{by \eqref{Gv-Gu}})\\
			&=& - \langle H^k_{\Gamma_{k}\Gamma_{k}}(\bv^k-\bu^k)_{\Gamma_k},(\bv^k-\bu^k)_{\Gamma_k} \rangle + 2o(\|\bv^k-\bu^k\|^2) ~~~~(\text{by \eqref{Newton-descent-property}})\\
			&\leq& - l_f  \|(\bv^k-\bu^k)_{\Gamma_k} \|^2 +o(\|\bv^k-\bu^k\|^2) ~~~~(\text{by \eqref{s-regular-1} or \eqref{s-regular}})\\
			&=& - l_f \|\bv^k-\bu^k \|^2 + 2 o(\|\bv^k-\bu^k\|^2)~~~~(\text{by \eqref{Gv-Gu}})\\
			&\leq&  -(l_f/{2}) \|\bv^k-\bu^k \|^2~~~~(\text{by   \eqref{vk-uk}}) \\
			&\leq& -\sigma \|\bv^k-\bu^k \|^2.~~~~(\text{by   $\sigma\in(0,l_f/2)$})
		\end{array}
	\end{eqnarray*}
	Overall, the Newton step is always admitted  for sufficiently large $k$.  {The proof is complete.\qed}
\end{proof}
Finally, we conclude that GPNA can converge quadratically for  SLCoRe and terminate within finite steps for SCoRe by the following theorem.
\begin{theorem} \label{theo07} Suppose $[X~Z]$ is $s$-regular. Then the sequence generated by GPNA with $\sigma\in(0,l_f/2)$ eventually converges to its limit quadratically for SLCoRe or within finitely many steps for SCoRe, namely, for sufficiently large $k$,
	\begin{eqnarray*}\label{4.9}
		\arraycolsep=1.4pt\def\arraystretch{1.5}
		\begin{array}{llll}
			\|\bt^{k+1}-\bt^{*}\| \leq  \frac{(1+L_f)^2C_f}{2l_f}\|\bt^{k}-\bt^*\|^{2},&~~\text{if}&~~ \ell=\ell_{log},\\
			\bt^{k+1}=\bt^{*},&~~\text{if}&~~ \ell=\ell_{lin}.
		\end{array}
	\end{eqnarray*} 
\end{theorem}
\begin{proof}We first  estimate $\|\bu^{k}-\bt^*\|$. Recalling \eqref{notation-k-1} that 
	\begin{eqnarray*} 
		\bu^k=\bt^k(\alpha_k) \in  \Pi_{\Sigma}(\bt^{k}- \alpha_k \nabla f(\bt^k)) 
	\end{eqnarray*}
	and $\Gamma_k=\Gamma(\bu^k)$, we have
	\begin{eqnarray*} 
		\bu^k_{\Gamma_k}=\bt^{k}_{\Gamma_{k}}-\alpha_k (\nabla f(\bt^{k}))_{\Gamma_k},~~~~ \bu^k_{\overline \Gamma_k}=0.
	\end{eqnarray*}
	This enables us to deliver that
	\begin{eqnarray} \label{fact-uk-x*}
		\arraycolsep=1.4pt\def\arraystretch{1.5}
		\begin{array}{llll}
			\|\bu^{k}-\bt^*\|  
			&=& \|\bt^{k}_{\Gamma_{k}}-\alpha_k (\nabla f(\bt^{k}))_{\Gamma_k} - \bt^*_{\Gamma_{k}}\| ~~~~(\text{by   $\bu^k_{\overline \Gamma_k}=\bt^*_{\overline \Gamma_k}=0 $ from \eqref{support-identify}})\\
			&=& \|\bt^{k}_{\Gamma_{k}}-\alpha_k (\nabla f(\bt^{k}))_{\Gamma_k} - \bt^*_{\Gamma_{k}}-\alpha_k(\nabla f(\bt^{*}))_{\Gamma_k})\| ~~~~(\text{by   \eqref{b*-gamma}})\\
			&\leq& \|\bt^{k}_{\Gamma_{k}}- \bt^*_{\Gamma_{k}}\|+\alpha_k \|(\nabla f(\bt^{k}))_{\Gamma_k} - (\nabla f(\bt^{*}))_{\Gamma_k})\|\\
			&\leq& (1+L_f)\|\bt^{k} - \bt^* \|. ~~~~~ (\text{by $0< \alpha_k\leq 1$ and \eqref{Lip-g}})  
		\end{array}
	\end{eqnarray}
	By Lemma \ref{quadratic-lemma} 2),  the Newton step is always admitted for sufficiently large $k$. Then direct calculations lead the following chain of inequalities,
	\begin{eqnarray*} 
		\arraycolsep=1.4pt\def\arraystretch{1.5}
		\begin{array}{llll}
			&& \|\bt^{k+1}-\bt^*\| = \|\bv^{k}-\bt^*\| \\
			&=& \|\bv^{k}_{\Gamma_{k}}-\bt^*_{\Gamma_{k}}\| ~~~~~~~~(\text{by   $\bv^k_{\overline \Gamma_k}=\bt^*_{\overline \Gamma_k}=0 $ from \eqref{support-identify}})\\
			&=& \|\bu^{k}_{\Gamma_{k}}-\bt^*_{\Gamma_{k}} - (H^k_{\Gamma_{k}\Gamma_{k}})^{-1} (\nabla f(\bu^{k}))_{\Gamma_k}\| &~~(\text{by   \eqref{Newton-descent-property}})\\
			&=& \|\bu^{k}_{\Gamma_{k}}-\bt^*_{\Gamma_{k}} - (H^k_{\Gamma_{k}\Gamma_{k}})^{-1} ((\nabla f(\bu^{k}))_{\Gamma_k}-(\nabla f(\bt^{*}))_{\Gamma_k})\| &~~(\text{by   \eqref{b*-gamma}})\\
			&\leq&(1/ l_f)  \|H^k_{\Gamma_{k}\Gamma_{k}}(\bu^{k}_{\Gamma_{k}}-\bt^*_{\Gamma_{k}} )- ((\nabla f(\bu^{k}))_{\Gamma_k}-(\nabla f(\bt^{*}))_{\Gamma_k})\| &~~(\text{by \eqref{s-regular-1} or \eqref{s-regular}})\\
			&\leq&(1/ l_f)  \|\nabla^2 f(\bu^k)(\bu^{k}-\bt^* )- (\nabla f(\bu^{k})-\nabla f(\bt^{*}))\| & \\
			&=&(1/ l_f)  \|\int_0^1 (\nabla^2 f(\bu^*+t(\bu^{k}-\bt^*))-\nabla^2 f(\bu^k)) (\bu^{k}-\bt^* )dt\| & \\ 
			&\leq&(1/ l_f)   \int_0^1 \|\nabla^2 f(\bu^*+t(\bu^{k}-\bt^*))-\nabla^2 f(\bu^k)\|\|\bu^{k}-\bt^*\|dt. 
		\end{array}
	\end{eqnarray*}
	Note that if $\ell=\ell_{lin}$, then $\nabla^2 f(\bu^*+t(\bu^{k}-\bt^*))=\nabla^2 f(\bu^k)=Q$. The above condition implies $\|\bt^{k+1}-\bt^*\|\leq 0$, namely, $\bt^{k+1}=\bt^*$. If $\ell=\ell_{log}$, then  above condition implies
	\begin{eqnarray*} 
		\arraycolsep=1.4pt\def\arraystretch{1.5}
		\begin{array}{llll}
			\|\bt^{k+1}-\bt^*\|  
			&\leq&(1/ l_f)  \int_0^1 C_f \| \bu^*+t(\bu^{k}-\bt^*) - \bu^{k}\|\|\bu^{k}-\bt^*\|dt &~~(\text{by  \eqref{Lip-hassian}})  \\
			&\leq&(C_f/ l_f) \|\bu^{k}-\bt^*\|^2  \int_0^1 (1-t)dt\\
			&=&(C_f /(2l_f))\|\bu^{k}-\bt^*\|^2.
		\end{array}
	\end{eqnarray*}
	which combining \eqref{fact-uk-x*} can make the conclusion immediately.  {The proof is complete.\qed}
\end{proof}

\section{Numerical experiments}\label{sec:num}

This section implements GPNA to solve SCL with synthetic datasets and real datasets. All numerical experiments are conducted by running MATLAB (R2018b) on an ideapad with   CPU @2.30GHz 2.40GHz and 4GB memory. Apart from the stopping criterion outlined in the algorithm, we also set the maximum number of iterations to 1000. We set $S=\varepsilon=0.0001, \epsilon =0.001, \alpha_0=1$ and $\gamma=0.5$. The initial point is chosen as $\bt^0=0$.

\subsection{ {SLCoRe model for discrete response variables}}
In this subsection, we  solve SCL with $\ell=\ell_{log}$, namely, SLCoRe. This model usually works well for the data with discrete response variables. In the sequel, we first present two testing examples, followed by the parameters' tuning for GPNA and its numerical comparisons with some benchmark methods on synthetic and real datasets.

\subsubsection{ {Test examples}}
Synthetic and real data are tested for SLCoRe.
\begin{example}[Synthetic data]\label{synthetic1} Similar to \cite{Bahm}, each sample $\bx_{i},i\in[n]$ in $X \in \mathbb{R}^{n\times p_1}$ is independently generated by an autoregressive process
	$$ x_{i(j+1)}=\theta   x_{ij}+\sqrt{1-\theta^{2}} c_{j} \quad \text { for all } j \in[p_{1}-1],$$
	with $x_{i1} \in \mathcal{N}(0,1), c_{j} \in \mathcal{N}(0,1)$ and $\theta \in [0,1)$ being the correlation parameter. Note that the larger $\theta$ is, the more correlated  two columns are. Let $Z= X + 0.01\cdot\Lambda$ with $\Lambda_{ij}\in\mathcal{N}(0,1)$. Therefore, for such an example, $p_1=p_2=:p$.  The sparse parameters $\bt_{1}\in \mathbb{R}^{p}$ and  $\bt_{2}\in \mathbb{R}^{p}$ have $s_{1}$ and $s_{2}$ nonzero entries  that are drawn independently from the standard Gaussian distribution, respectively. Finally, response $\by \in\{0,1\}^{n}$ is  randomly  generated from the Bernoulli distribution with
	$$\operatorname{Prob}\{y_i=0 \mid \bx_i,\bz_i\}=\frac{1}{2}\left[\frac{1}{ 1+\exp \left(-\langle \bx_i, \bt_{1} \rangle \right)} +\frac{1}{ 1+\exp \left(-\langle \bz_i, \bt_{2} \rangle \right)} \right].$$
\end{example} 
\begin{example}[Real data]\label{real1} Two real datasets are taken into account. They are the alcohol dependence data with $n =46$, $p_{1} =500$ and $p_{2} =300$  \cite{Zhangh}\footnote{Available at \url{https://github.com/cran/CVR/blob/master/data/alcohol.rda} }   and Diffuse large B-cell lymphoma (DLBCL) data with $n =203$, $p_{1} =17350$ and $p_{2} =386165$  \cite{Lenz} \footnote{Available at \url{http://www.ncbi.nlm.nih.gov/geo/query/acc.cgi?acc=GSE11318}}. All datasets are feature-wisely scaled to $[-1,1]$.  
\end{example} 

To evaluate the performance of one method, we report  the CPU time (in seconds), the classification error rate (CER) \cite{Wang} and the canonical correlation value (CCV) defined by
\begin{eqnarray*}
	\text{CER}:=\frac{\|\operatorname{sign} ( X\bt_{1} )-\by\|_0+ \|\operatorname{sign} ( Z\bt_{2} )-\by\|_0}{n},~~
	\text{CCV}:=\frac{\|X\bt_1-Z\bt_2\|}{n},
\end{eqnarray*} 
where $\bt=(\bt_{1};\bt_{2})$ is the solution obtained by one method and  $({\rm sign}(\bx))_i=1$ if $x_i > 0$ and $({\rm sign}(\bx))_i=0$ otherwise for $i\in[n]$. Note that the smaller CER (or the smaller CCV or the shorter CPU time) the better performance. 
\begin{figure}[!th]
	\centering
	\includegraphics[width=.495\textwidth]{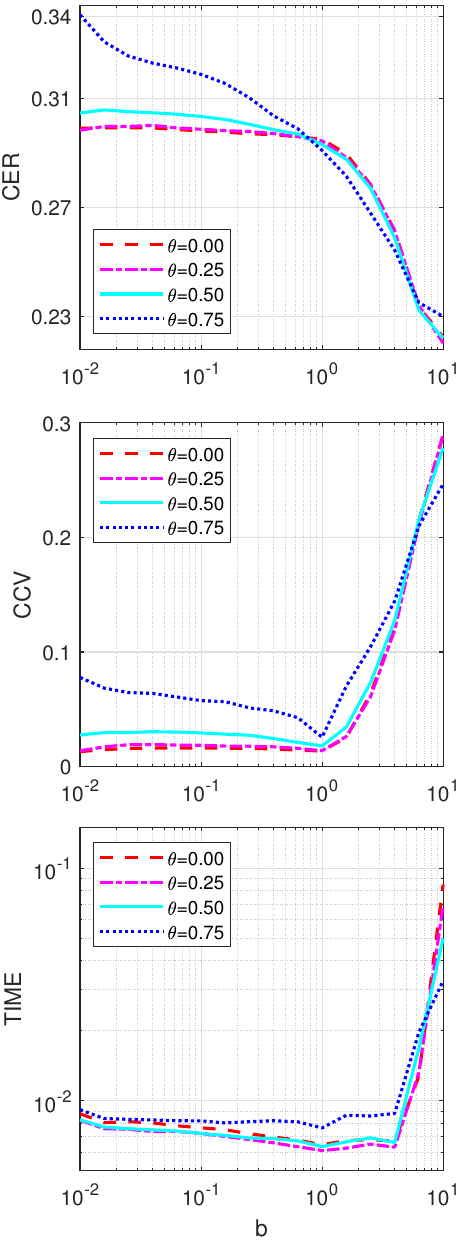}  
	\includegraphics[width=.495\textwidth]{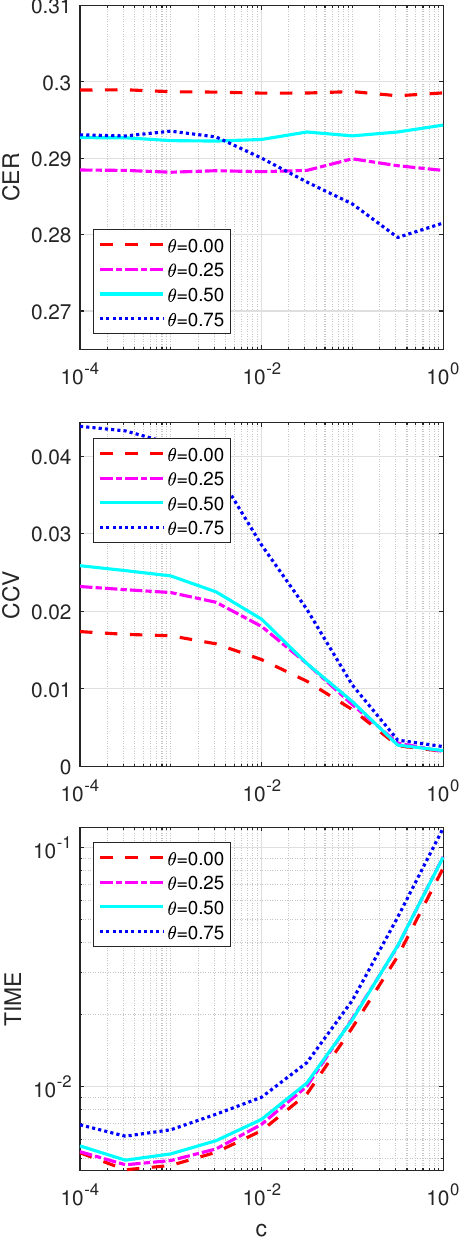}\vspace{-3mm}
	\caption{Effect of  $b$ and $c$ for Example \ref{synthetic1}.}\label{fig:effect-c}	
\end{figure}

\subsubsection{Sensitivity analysis} We now implement GPNA to see its performance under different choices of $(a,b,c,s_1,s_2)$ .

{\bf (a) Effect of $(a,b,c)$.} Recall that there are three parameters $(a,b,c)$ involved in problem \eqref{005}. We fix $a=1, c=0.01$ but vary $b\in[0.01,10]$ to see the effect of $b$ and  fix $a=1, b=1$ but change $c\in[0.0001,1]$ to see the effect of $c$. The average results over 100 instances for Example \ref{synthetic1} are presented in  {Fig. \ref{fig:effect-c}}, where $n=200,p=2000$ and $s_1=s_2=20$. 

When $a$ and $c$ are fixed, from the three above sub-figures in  {Fig. \ref{fig:effect-c}}, one can observe that CER is declining steadily when $b\in[0.01,1)$ but dramatically  when $b\in[1,10]$. However, the best choice of $b$ for CCV and CUP time is $b=1=a$. Therefore, for Example \ref{synthetic1} with $p_1=p_2$ and $s_1=s_2$, the best option to set $a$ and $b$ should be $a=b$.

When $a$ and $b$ are fixed, from the three bottom sub-figures in  {Fig. \ref{fig:effect-c}}, it can be clearly seen that the larger values of $c$, the smaller CCV and longer CPU time. One can observe that the variance of $c\in[0.0001,0.01]$ do not influence CER significantly.

We test some other choices and find the following options for $(a,b,c)$ that allows GPNA to render desirable overall performance:
$$a=\frac{s_1}{s_1+s_2},~~b=\frac{s_2}{s_1+s_2},~~c=\frac{1}{s_1+s_2}.$$
Therefore, in the following numerical experiments, we fix $a,b,c$ as above choices if no additional information is provided.

{\bf (b) Effect of $(s_1,s_2)$.} To see the effect of  $s_1$ and $s_2$, we  choose both $s_1$ and $s_2 $ from $ \{5,10,\cdots,40\}$. The average results of GPNA for Example \ref{synthetic1} are shown in  {Fig. \ref{fig:effect-s1s2}} where $n=200,p=2000, \theta=0.5$. The figure demonstrates that the larger $s_1$ or $s_2$ the higher values of CER, leading to better performance. Moreover, the closer between  $s_1$ and $s_2$ is, the smaller CCV is. 

\begin{figure}[!th]
	\centering
	\includegraphics[width=.475\textwidth]{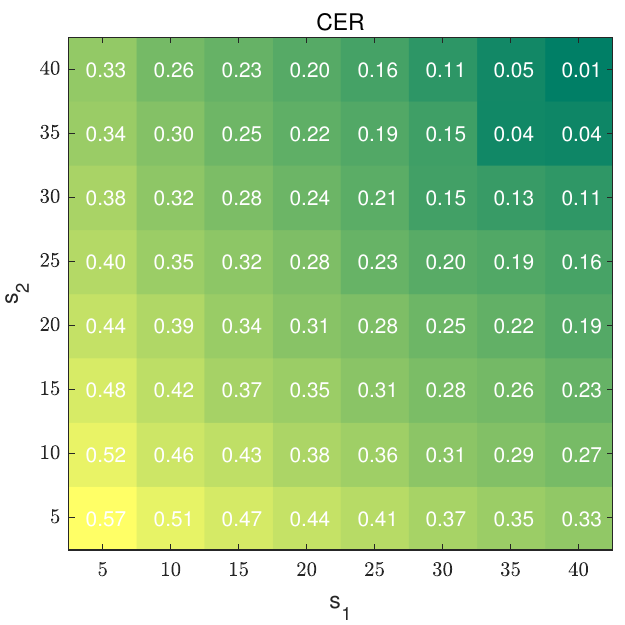} 
	\includegraphics[width=.475\textwidth]{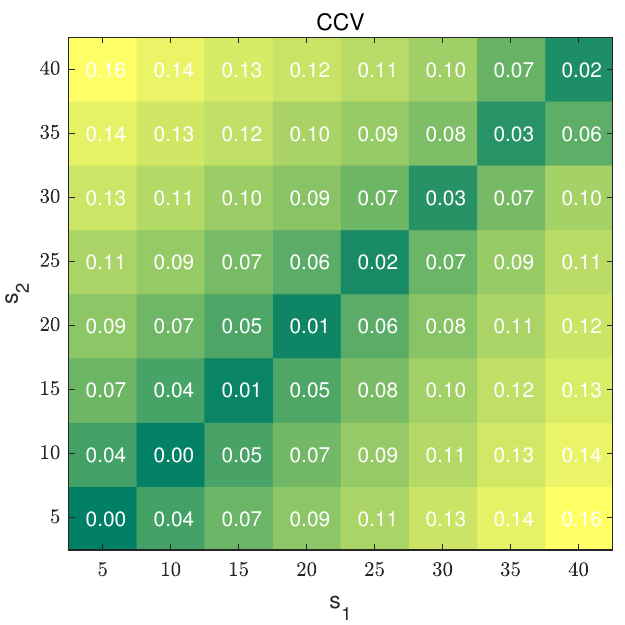}\vspace{-3mm}
	\caption{Effect of  $s_1$ and $s_2$ for Example \ref{synthetic1}.}\label{fig:effect-s1s2}	
\end{figure}

\subsubsection{ {Effectiveness}}

To illustrate the effectiveness of our proposed model SLCoRe as well as the method GPNA, several alternative approaches are selected. They are SCoRe \cite{GRO}, GPGN \cite{Wang}, GraSP \cite{Bahm}, IIHT \cite{Pan} and NTGP \cite{Yuan2017}. The first one is used to solve the SCoRe, which can be used to illustrate that SLCoRe is a better model than SCoRe for the discrete response variables. GPGN, GraSP,  IIHT and NTGP  solve the sparse logistic regression that merges two datasets into a single one, which can be used to highlight the advantage of the model SLCoRe for two interrelated  datasets. 

{\bf (c) Comparison for Example \ref{synthetic1}.}   For simplicity, we fix $n=1000,p=10000$ while choose $\theta\in\{0,0.5,0.8\}$ and $s_1,s_2\in\{200,300,500\}$. For each case of $(\theta,s_1,s_2)$, we test 100 instances and report the average results of GPNA, SCoRe, GPGN, IIHT, GraSP and NTGP. Some comments on the reported data in Tables \ref{Table1} and \ref{Table1-1} can be made.

\begin{table}[!htbp]
	\renewcommand{\arraystretch}{1}\addtolength{\tabcolsep}{1.8pt}
	\caption{Comparison of the results for Example \ref{synthetic1}. }\label{Table1}
	\centering
	\begin{tabular}{ll|ccc|ccc|ccc}
		\hline
		&&\multicolumn{3}{c}{CER}& \multicolumn{3}{|c|}{CCV}&\multicolumn{3}{c}{TIME}\\\cline{3-11}
		&  &\multicolumn{3}{c}{$s_2$}& \multicolumn{3}{|c|}{$s_2$}&\multicolumn{3}{c}{$s_2$}\\
		$s_1$&Algs. & 200&300& 500 &200&300& 500&200&300& 500\\
		\hline
		\multicolumn{11}{c}{\textbf{$\theta=0$}}\\\hline		
		{$200$}&
		GPNA& \textbf{0.013} &\textbf{0.016} & \textbf{0.022}&\textbf{0.040} &\textbf{0.074} &\textbf{0.086}&\textbf{00.5} &\textbf{00.6} &\textbf{00.5} \\
		&SCoRe& 0.410 &0.452 & 0.398&0.231 &0.332 &0.354 &15.5 &17.1 &18.1 \\
		&IIHT& 0.396 & 0.259&0.263 &0.286& 0.512&0.519 &01.5 &01.7 &01.5 \\
		&GraSP& 0.224 & 0.363&0.245 &0.384& 0.467&0.596 &02.9 &02.2 &01.6 \\
		&GPGN& 0.117  &0.144  &0.131  &0.486 & 0.643 & 0.678 &00.7  &00.7  & 00.8 \\
		&NTGP& 0.128  &  0.137&0.143  &0.573 & 0.586 & 0.697 &01.1  & 01.2 &01.2  \\\hline
		{$300$}
		&GPNA&\textbf{0.012}  &\textbf{0.000} &\textbf{0.000} &\textbf{0.084} &\textbf{0.032} & \textbf{0.062}&\textbf{00.4} &\textbf{00.5} &\textbf{00.5} \\
		&SCoRe&0.391  &0.384 &0.423 &0.382 &0.518 &0.521 &20.3 &24.2 &23.6 \\
		&IIHT& 0.241 & 0.256&0.407 &0.476& 0.561&0.869 &01.5 &01.7 &01.5 \\
		&GraSP& 0.237 & 0.266&0.304 &0.561& 0.627&0.922&03.8 &02.1 &01.6 \\
		&GPGN& 0.134  & 0.107 & 0.096 &0.558 &0.734  &0.877  & 00.7 &00.7  & 00.8 \\
		&NTGP& 0.118  & 0.142 &0.153  &0.621 & 0.727 & 0.973 & 01.4 &01.5  & 01.4 \\\hline
		{$500$}
		&GPNA& \textbf{0.024} & \textbf{0.000}&\textbf{0.000} &\textbf{0.087} & \textbf{0.061}&\textbf{0.014} &\textbf{00.5} &\textbf{00.7} &\textbf{00.6} \\
		&SCoRe&0.425  &0.459 &0.480 &0.231 &0.242 &0.318 &20.3 &24.6 &26.4 \\
		&IIHT& 0.323 & 0.413&0.328 &0.396& 0.461&0.469&01.5 &01.3 &01.5 \\
		&GraSP& 0.243 & 0.261&0.252 &0.853& 0.886&0.877 &01.6 &01.3 &01.1 \\
		&GPGN& 0.126  & 0.109 & 0.115 &0.878 & 0.974 &0.963  &00.8  &00.9  &00.9  \\
		&NTGP& 0.135  &0.183  & 0.167 &0.931 &0.924  &0.987  & 01.4 &01.5  &01.6  \\
		\hline
		\multicolumn{11}{c}{\textbf{$\theta=0.5$}}\\\hline
		{$200$}
		&GPNA&\textbf{0.014}  &\textbf{0.021} &\textbf{0.023} &\textbf{0.050} &\textbf{0.086} &\textbf{0.087} &\textbf{00.4} &\textbf{00.4} &\textbf{00.5} \\
		&SCoRe&0.423  & 0.398& 0.451&0.213 &0.252 &0.385 &16.7 &18.9 &21.1 \\
		&IIHT& 0.264 & 0.253&0.246 &0.319& 0.478&0.491 &02.1 &01.5 &02.0 \\
		&GraSP& 0.243 & 0.258&0.251 &0.343& 0.437&0.553 &04.7 &02.6 &01.5 \\
		&GPGN& 0.134  & 0.118 & 0.156 &0.461 & 0.586 & 0.672 &00.6  &00.7  &00.7  \\
		&NTGP& 0.144  &0.152  &0.153  & 0.429&0.536  & 0.543 & 01.5 & 01.6 & 01.8 \\\hline
		{$300$}
		&GPNA&\textbf{0.017 } &\textbf{ 0.000}&\textbf{0.000} &\textbf{0.086} &\textbf{0.032} &\textbf{0.043} &\textbf{00.4} &\textbf{00.6} &\textbf{00.7}\\
		&SCoRe& 0.423 &0.384 &0.366 &0.247 &0.342 &0.425 & 21.5&23.7 &24.8 \\
		&IIHT& 0.246 & 0.273&0.282 &0.324& 0.363&0.513 &01.7 &01.9 &01.5 \\
		&GraSP& 0.257 & 0.253&0.271 &0.337& 0.472&0.438 &03.2 &01.9 &01.4 \\
		&GPGN&0.145 &0.157  & 0.138 &0.512 & 0.466 &0.539  &00.7  &00.8  &00.8  \\
		&NTGP& 0.148  &0.157  &0.143  &0.384 & 0.473 & 0.614 & 01.6 &01.7  &01.9  \\\hline
		{$500$}
		&GPNA&\textbf{ 0.018} & \textbf{0.000}& \textbf{0.000}& \textbf{0.089}&\textbf{0.071} &\textbf{0.020} &\textbf{00.5} &\textbf{00.7} &\textbf{00.7} \\
		&SCoRe&0.443  &0.483 &0.456 &0.343 &0.462 &0.437 &18.4 &22.9 &28.4 \\
		&IIHT& 0.239 & 0.252&0.399 &0.478& 0.526&0.854 &01.8 &01.6 &01.7 \\
		&GraSP& 0.258 & 0.264&0.285 &0.523& 0.528&0.694 &01.5 &01.3 &01.1 \\
		&GPGN& 0.157  &0.127  &0.148  & 0.633&0.579  & 0.715 & 00.7 &\textbf{00.7} & 00.8 \\
		&NTGP& 0.128  & 0.137 &0.153  &0.584 & 0.162 & 0.849 & 01.7 &01.6  &01.7  \\
		\hline
	\end{tabular}
\end{table}

\begin{table}[!th]
	\renewcommand{\arraystretch}{1}\addtolength{\tabcolsep}{2.3pt}
	\caption{Comparison of the results for Example \ref{synthetic1}. }\label{Table1-1}
	\centering
	\begin{tabular}{ll|ccc|ccc|ccc}
		\hline
		&&\multicolumn{3}{c}{CER}& \multicolumn{3}{|c|}{CCV}&\multicolumn{3}{c}{TIME}\\\cline{3-11}
		&  &\multicolumn{3}{c}{$s_2$}& \multicolumn{3}{|c|}{$s_2$}&\multicolumn{3}{c}{$s_2$}\\
		$s_1$&Algs. & 200&300& 500 &200&300& 500&200&300& 500\\
		\hline
		\multicolumn{11}{c}{\textbf{$\theta=0.8$}}\\\hline
		{$200$}
		&GPNA&\textbf{0.055}  &\textbf{0.051} & \textbf{0.060}&\textbf{0.085} &\textbf{0.104} &\textbf{0.105} &\textbf{00.4} &\textbf{00.5} &\textbf{00.5} \\
		&SCoRe&0.491  &0.473 &0.432 &0.252 &0.344 &0.335 &14.7 &19.4 &21.6 \\
		&IIHT& 0.282 & 0.260&0.251 &0.274& 0.417&0.423&02.0 &02.1 &01.9 \\
		&GraSP& 0.301 & 0.253&0.268 &0.284& 0.349&0.464 &06.4 &04.1 &03.5\\
		&GPGN&0.167 &0.137  & 0.142 & 0.433& 0.512 &0.641  &00.6  & 00.7 & 00.7 \\
		&NTGP&  0.211 & 0.176 & 0.189 &0.343 & 0.487 & 0.622 & 01.7 & 01.7 &01.8  \\\hline
		{$300$}
		&GPNA&\textbf{0.052}  &\textbf{0.000} &\textbf{0.000} &\textbf{0.105} &\textbf{0.059} &\textbf{0.093} & \textbf{00.4}& \textbf{00.5}&\textbf{00.5} \\
		&SCoRe&0.435  & 0.412&0.397 &0.334 &0.338 &0.396 &16.4 &20.3 &23.7 \\
		&IIHT& 0.268 & 0.271&0.257 &0.434& 0.475&0.587&01.6 &01.7 &01.4 \\
		&GraSP& 0.276 & 0.284&0.245 &0.376& 0.433&0.639 &03.9 &03.2 &01.8 \\
		&GPGN& 0.154  &0.138  &0.165  &0.533 &0.537  & 0.626 & 00.7 &00.7  &00.8  \\
		&NTGP& 0.204  &0.225  & 0.188 &0.526 &0.491  &0.654  & 01.6 & 01.7 &01.9  \\\hline
		{$500$}
		&GPNA& \textbf{0.061} & \textbf{0.000}&\textbf{0.000} &\textbf{0.108} & \textbf{0.085}&\textbf{0.031} &\textbf{00.5} &\textbf{00.5} &\textbf{00.6} \\
		&SCoRe&0.457  &0.423 &0.382 &0.324 &0.356 &0.431 &18.4 &22.8 &28.7 \\
		&IIHT& 0.255 & 0.258&0.266 &0.527& 0.529&0.912 &01.6 &01.4 &02.0 \\
		&GraSP& 0.239 & 0.254&0.263 &0.518& 0.538&0.883 &02.4 &01.8 &01.4 \\
		&GPGN& 0.154  &0.162  &0.166  & 0.634&0.568  & 0.942 & 00.7 & 00.8 &00.8  \\
		&NTGP& 0.173  & 0.213 &0.188  & 0.638& 0.652 &0.875  & 01.7 & 01.9 & 01.8 \\
		\hline
	\end{tabular}
\end{table}

Regarding CER,  GPNA achieves the minimum  values compared with other methods regardless of the sparsity and correlation how to change. The error rate of the other five methods is more than 10\% for the case of two data sets. Moreover, CERs obtained by  GPNA, GPGN, IIHT, GraSP and NTGP are smaller than SCoRe, which indicates that $\ell_{log}$ is more advantageous than $\ell_{lin}$ for the discrete responses.

Regarding CCV,  GPNA delivers tiny  values, which shows that there is a high correlation between the two datasets.  Although  SCoRe can also reveal the relationship between two datasets, the result is not as good as GPNA. Nevertheless, they both perform smaller CCVs than GPGN, IIHT, GraSP and NTGP since the latter four methods solve the model that ignores the   relationship between two datasets.

Regarding CPU time, it is obvious that GPNA
is the fastest and the calculations take less than a second for all scenarios. By contrast, the other methods need much longer time, especially for larger sparsity, with SCoRe taking 28 seconds, which is 47 times longer than GPNA.

{\bf (d) Comparison for Example \ref{real1}.}  This part reports the numerical comparisons of GPNA, SCoRe, GPGN, IIHT, GraSP  and NTGP for analysing two real datasets.

We first apply our method to jointly analyze methylation and gene expression data in an alcohol dependence study \cite{Zhangh}.   SLCoRe can be used to identify the canonical variates from DNA methylation (corresponding to $X$) and gene expression  (corresponding to $Z$) supervised by the phenotypical information, e.g., alcohol use disorder (AUD), which is observed as a binary indicator variable  $\by$. In this study, genome-wide DNA methylation levels and genome-wide expression levels of genes are quantified for $n =46$ European Australians. Similar to \cite{Luoc}, we choose top $p_{1} =500$ CpG sites and $p_{2} =300$ genes associated with AUD.

\begin{table}[!htbp]
	\renewcommand{\arraystretch}{.88}\addtolength{\tabcolsep}{7pt}
	\caption{Comparison of the results for Example \ref{real1}.}\label{Table2}
	\centering
	\begin{tabular}{l|ll|lccc|cc}
		\hline
		&  & & \multicolumn{3}{c}{Training} && \multicolumn{2}{c}{Testing}   \\\cline{4-6}\cline{7-9}
		&$s_1$ &$s_2$&CER &CCV &TIME(s)&&  CER& CCV  \\\cline{1-9}
		\multicolumn{9}{c}{AUD}\\\hline
		SCoRe& & &  0.617&0.200 &001.6&&0.582& 0.278 \\\hline
		GPNA&20&10&\textbf{0.025}&\textbf{0.004}&\textbf{000.2}&&\textbf{0.004}&\textbf{0.005}\\
		&20&20&\textbf{0.020}&\textbf{0.009} & \textbf{000.2}&&\textbf{0.002}&\textbf{0.007} \\ 
		&35&20&\textbf{0.018}& \textbf{0.008}&\textbf{000.2}&&\textbf{0.002}&\textbf{0.012} \\
		&35&35&\textbf{0.017}&\textbf{0.007} &\textbf{000.3}&&\textbf{0.000}& \textbf{0.005}\\	\hline
		IIHT&20&10&0.525&0.248 &001.6&&0.480&0.890\\
		&20&20&0.472&0.251 &001.6&&0.530&0.893\\
		&35&20&0.455& 0.251&001.6&&0.463&0.889\\
		&35&35&0.466&0.253&001.7&&0.428&0.871\\
		GraSP&20&10&0.528& 0.338&001.5&&0.410&0.932\\
		&20&20&0.443&0.336&001.3&&0.500&0.919\\
		&35&20&0.487&0.328&001.7&&0.422&0.922\\
		&35&35&0.482&0.334&001.7&&0.338&0.916\\\hline
		GPGN&20&10&0.243 &0.365  &\textbf{000.2} &&0.334 & 0.974\\
		&20&20&0.284 &0.378 &000.3 &&0.347 &0.868 \\
		&35&20&0.233&0.469&000.3&&0.346&0.884\\
		&35&35&0.215&0.478&\textbf{000.3}&&0.317&0.967\\\hline
		NTGP&20&10&0.556 & 0.595 &000.3 &&0.420 & 0.863\\
		&20&20&0.524 &0.553 &000.3 &&0.376 & 0.761\\
		&35&20&0.488&0.528&000.3&&0.397&0.837\\
		&35&35&0.472&0.557&\textbf{000.3}&&0.385&0.868\\
		\cline{1-9}		
		\multicolumn{9}{c}{DLBCL}\\\hline
		SCoRe& & & 0.753& 0.592&070.4 && 0.682 &0.634 \\\hline
		GPNA&50&50&\textbf{0.054} & \textbf{0.036}&\textbf{000.3}&&\textbf{0.024}&\textbf{0.029} \\
		&50&100&\textbf{0.034}& \textbf{0.067}&\textbf{000.3}&&\textbf{0.039} & \textbf{0.085}\\
		&100&100&\textbf{0.000}&\textbf{0.024} &\textbf{000.3}&&\textbf{0.001} &\textbf{0.022} \\
		&100&150&\textbf{0.000}&\textbf{0.017} &\textbf{000.3 }&&\textbf{0.002} &\textbf{0.014} \\	\hline	
		IIHT&50&50&0.471&0.796 &042.5&&0.464&0.732\\
		&50&100&0.458&0.763 &043.6&&0.483&0.746\\
		&100&100&0.488&0.743 &046.3&&0.462&0.737\\
		&100&150&0.482&0.737 &047.6&&0.455&0.739\\\hline
		GraSP&50&50&0.456&0.854 &235.4&&0.472&0.861\\
		&50&100&0.458& 0.846&254.7&&0.483&0.867\\
		&100&100&0.432&0.852 &228.6&&0.457&0.854\\
		&100&150&0.427& 0.848&233.2&&0.463&0.851\\\hline
		GPGN&50&50&0.408 &0.973  &000.9 &&0.487 &0.832 \\
		&50&100&0.384 & 0.972&000.9 &&0.453 &0.731 \\
		&100&100&0.387&0.881&001.1&&0.473&0.848\\
		&100&150&0.395&0.956&001.1&&0.434&0.907\\\hline
		NTGP&50&50&0.421 & 0.834 &038.2 &&0.428 & 0.911\\
		&50&100& 0.452&0.786 &040.3 &&0.478 & 0.841\\
		&100&100&0.478&0.879&041.8&&0.503&0.812\\
		&100&150&0.474&0.934&042.4&&0.433&0.865\\
		\hline
	\end{tabular}
\end{table}

We use a random splitting procedure to compare the six methods. At each split, 10 observations are randomly chosen as the testing data and the remaining 36 observations are the training data. The random splitting is repeated 100 times. We choose different sparsity and the average results are reported in Table \ref{Table2} and show  the better behaviour of GPNA since it obtains lower CER (meaning better predictions),  smaller CCV and runs much faster.

We next deal with a higher dimensional real dataset DLBCL \cite{Lenz}. It comprises  of $n=203$ patients, each of which has $p_1=17350$ gene expression  and $p_2=386165$ copy numbers. We  fixate on the case where $\by$ is a binary variable indicating the survival or death or the cancer subtype. Again, the 203 samples are split into 153 ones as the training set and 50 ones as the testing set. The random splitting is repeated 100 times. Similar phenomenon to AUD data can be observed for DLBCL in Table \ref{Table2}, showing the better performance of GPNA.

\subsection{ {SCoRe model for continuous response variables}}
In the subsequent numerical experiments, we focus on SCL with $\ell=\ell_{lin}$, namely, SCoRe. This model is proper  for the  data with  continuous response variables. For such a model, we also do parameters' tuning for GPNA and get similar observations to that for SLCoRe. Therefore, we keep the same setting of parameters as previous examples for GPNA.

\subsubsection{ {Test examples}}
Again, synthetic and real data are tested for SCoRe.
\begin{example}[Synthetic data]\label{synthetic2} The sample data $X$ and $Z$ as well as the sparse parameters $\bt_{1}\in \mathbb{R}^{p}$ and  $\bt_{2}\in \mathbb{R}^{p}$ are generated the same as Example \ref{synthetic1}, while the response $\by $ is  generated by
	$\by=  (X\bt_{1}+Z\bt_{2})/2.$
\end{example} 
\begin{example}[Real data]\label{real2} Two real datasets are taken into consideration. They are the body mass index (BMI) of mouse data with $n =294$, $p_{1} =163$ and $p_{2} =215$   \cite{Wangs}\footnote{Available at \url{https://github.com/cran/CVR/blob/master/data/mouse.rda} }  and DLBCL data. All datasets are feature-wisely scaled to $[-1,1]$.  
\end{example} 

To evaluate the performance of one method, we report  the CPU time (in seconds), the mean square error (MSE) and CCV defined by
\begin{eqnarray*}
	\text{MSE}:=\frac{\|\by-X\bt_1\|+\|\by-Z\bt_2\|}{n},~~~~	\text{CCV}:=\frac{\|X\bt_1-Z\bt_2\|}{n},
\end{eqnarray*} 
where $\bt=(\bt_{1};\bt_{2})$ is the solution obtained by one method.

\subsubsection{ {Effectiveness}}

Besides three aforementioned methods SCoRe, GraSP, IIHT, we also select two additional methods  SP \cite{Dai} and LNA \cite{Zhao2021} for comparisons. Again, GraSP, IIHT,  SP and LNA are solving the problem without consider the interrelationship between two datasets. 

{\bf (e) Comparison for Example \ref{synthetic2}.} We first compare GPNA  with the other five methods for Example \ref{synthetic2}. For simplicity, we fix $n=2000,p=6000$ while choose $\theta\in\{0,0.5\}$ and $s_1,s_2\in\{100,200,500\}$. For each case of $(\theta,s_1,s_2)$, we test 100 instances and report the average results of GPNA, SCoRe, IIHT, GraSP, SP and LNA.  Some comments on the data in Table \ref{Table4} can be made. 

\begin{table}[!htbp]
	\renewcommand{\arraystretch}{1}\addtolength{\tabcolsep}{2pt}
	\caption{Comparison of the results for Example \ref{synthetic2}. }\label{Table4}
	\centering
	\begin{tabular}{ll|ccc|ccc|ccc}
		\hline
		&&\multicolumn{3}{c}{MSE}& \multicolumn{3}{|c|}{CCV}&\multicolumn{3}{c}{TIME}\\\cline{3-11}
		&  &\multicolumn{3}{c}{$s_2$}& \multicolumn{3}{|c|}{$s_2$}&\multicolumn{3}{c}{$s_2$}\\
		$s_1$&Algs. & 100&200& 500 &100&200& 500&100&200& 500\\
		\hline
		\multicolumn{11}{c}{\textbf{$\theta=0$}}\\\hline
		{$100$}&
		GPNA& \textbf{0.083} &\textbf{0.097} & \textbf{0.171}&\textbf{0.015} &\textbf{0.081}& \textbf{0.169}&\textbf{00.3} &\textbf{00.4} &\textbf{00.4} \\
		&SCoRe& 0.243 &0.268 & 0.284&0.031 &0.092 &0.201 &16.5 &17.6 &19.3\\
		&IIHT& 0.157  &0.189  &0.267  &0.155 &0.189  &0.245  & 01.3 &01.9  &07.3  \\
		&GraSP& 0.232  &0.173  & 0.322 &0.163 & 0.211 & 0.258 & 14.4 & 18.5  &33.6  \\
		&SP& 0.226  &0.255  & 0.364 & 0.160&0.185  &0.263  &01.4  &01.9  &21.1  \\
		&LNA& 0.167  & 0.173  & 0.247  & 0.174 & 0.177  & 0.223  & 00.7  & 00.8  & 01.4  \\\hline
		{$200$}
		&GPNA&\textbf{0.097}  &\textbf{0.115} &\textbf{0.161} &\textbf{0.081} &\textbf{0.010} & \textbf{0.139}&\textbf{00.4} &\textbf{00.5} &\textbf{00.5} \\
		&SCoRe&0.214  &0.277 &0.286 &0.082 &0.093 &0.175 &19.3 &19.2 &21.7\\
		&IIHT&0.207   &0.227  &0.287  &0.195 & 0.223 &0.240  &01.9  &03.2 &09.8  \\
		&GraSP& 0.211  & 0.252 & 0.324 & 0.243 & 0.245 &0.251 & 21.7 & 27.2 &47.4  \\
		&SP& 0.243  &0.309  & 0.394& 0.174&0.223  &0.271  &01.8 &03.1  &24.6  \\
		&LNA& 0.216  & 0.236  & 0.317  &0.225  &0.244   &  0.238 &00.8   & 00.9  & 01.6  \\\hline
		{$500$}
		&GPNA& \textbf{0.165} & \textbf{0.148}&\textbf{0.182} &\textbf{0.163} & \textbf{0.142}&\textbf{0.063} &\textbf{00.4} &\textbf{00.4} &\textbf{00.6} \\
		&SCoRe&0.291  &0.318 &0.285 &0.184 &0.147 &0.185 &17.4 &21.4 &23.7 \\
		&IIHT&0.288   &0.286  &0.367  &0.243 &0.235  &0.282 &08.6  & 11.7 & 16.4 \\
		&GraSP& 0.317  & 0.264 & 0.334 & 0.259& 0.221 &0.273  & 35.6 &50.2  &87.6  \\
		&SP& 0.351  &0.411  & 0.476& 0.253&0.288  &0.252  &16.5 &39.8  &56.1  \\
		&LNA&0.256   &0.342   & 0.329  & 0.285 & 0.283  & 0.264  &01.5   &01.7   & 02.6  \\
		\hline
		\multicolumn{11}{c}{$\theta=0.5$}\\\hline	
		{$100$}&
		GPNA& \textbf{0.094} &\textbf{ 0.096}&\textbf{0.188} &\textbf{0.015} & \textbf{0.082}&\textbf{0.172} &\textbf{00.3} &\textbf{00.6} &\textbf{00.9} \\
		&SCoRe&0.242  &0.265 &0.267 &0.144&0.167 &0.184 &16.7 &19.4 &22.3 \\
		&IIHT& 0.172  &0.196  &0.285  &0.163 &0.182  &0.241 & 01.8 & 02.5 & 11.5 \\
		&GraSP& 0.187  &0.224  & 0.273 &0.152 & 0.187 & 0.252 & 18.3 & 24.2  &32.6  \\
		&SP& 0.224 &0.252  & 0.377& 0.159&0.183  &0.269 &01.3 &01.9  &27.4 \\
		&LNA& 0.213  & 0.189  & 0.254  & 0.171 & 0.176  & 0.253  &00.6   &00.8   & 01.5  \\\hline
		{$200$}&
		GPNA&\textbf{0.095}  &\textbf{0.117} &\textbf{0.168} &\textbf{0.088} &\textbf{0.031} &\textbf{0.137} &\textbf{ 00.4}& \textbf{00.6}&\textbf{00.8} \\
		&SCoRe&0.212  & 0.224&0.281 &0.158 &0.145 &0.174 &18.8 &20.5 &23.1 \\
		&IIHT&0.196   &0.233  &0.317  &0.203 &0.218  &0.263 &02.6  &03.9  & 11.9 \\
		&GraSP& 0.248  & 0.236 & 0.339 & 0.199 & 0.253 &0.256 & 21.7 & 35.3 &47.8  \\
		&SP& 0.289 &0.322  & 0.368& 0.214&0.217  &0.282  &01.8 &03.5  &29.7 \\
		&LNA& 0.226  & 0.265  & 0.287  &0.248  & 0.255  & 0.252  & 00.9  &00.9   & 01.7  \\\hline
		{$500$}&
		GPNA& \textbf{0.187} & \textbf{0.167}&\textbf{0.182} &\textbf{0.183} & \textbf{0.129}&\textbf{0.056} &\textbf{00.6} &\textbf{00.7} &\textbf{00.7} \\
		&SCoRe&0.287 &0.245 &0.272 &0.195 &0.146 &0.122 &20.4 &21.8 &23.3 \\
		&IIHT& 0.275  &0.315  &0.346  &0.242 & 0.264 &0.144  &14.4 &21.7  &  24.3  \\
		&GraSP& 0.244  & 0.337 & 0.386 & 0.296& 0.302 &0.221  & 42.2 &51.7  &70.4  \\
		&SP& 0.358 &0.402  & 0.461& 0.263&0.278  &0.266  &21.8 &32.3 &57.4 \\
		&LNA& 0.253  &  0.296 & 0.306  & 0.314 &0.267   & 0.278  & 01.6  & 01.9  &  02.9 \\
		\hline
	\end{tabular}
\end{table}

Regarding MSE,  GPNA achieves the smallest values in comparison with the other methods regardless of how the sparsity levels $s_1,s_2$ and correlation parameter $\theta$ change.  Once again, GPNA produces relatively small CCVs, which indicates that there is a high correlation between the two datasets. By contrast, since IIHT, GraSP, SP and LNA do not take the correlation into account, their generated CCVs are higher than these by GPNA and SCoRe.  It can be clearly seen that GPNA runs the fastest, such as 0.6 seconds consumed when $s_1=s_2=500,~\theta=0$ v.s. 23.7, 16.4, 87.6, 56.1 and 2.6 seconds  by the other five methods. 

\begin{table}[!htbp]
	\renewcommand{\arraystretch}{.88}\addtolength{\tabcolsep}{7pt}
	\caption{Comparison of the results for Example \ref{real2}.}\label{Table5}
	\centering
\begin{tabular}{l|ll|lccc|cc}
		\hline
		&  & & \multicolumn{3}{c}{Training} && \multicolumn{2}{c}{Testing}   \\\cline{4-9}
		&$s_1$ &$s_2$&MSE &CCV &TIME(s)&&  MSE& CCV  \\\cline{1-9}
		\multicolumn{9}{c}{ Mouse}\\\hline
		SCoRe& & &  0.423&0.183 &001.2&&0.386& 0.172 \\\hline
		GPNA&20&10&\textbf{0.196}&\textbf{0.121}&\textbf{000.1}&&\textbf{0.256}&\textbf{0.151}\\
		&20&20&\textbf{0.174}&\textbf{0.120} & \textbf{000.1}&&\textbf{0.223}&\textbf{0.136} \\ 
		&40&20&\textbf{0.141}& \textbf{0.103}&\textbf{000.1}&&\textbf{0.184}&\textbf{0.126} \\
		&40&40&\textbf{0.125}&\textbf{0.088}&\textbf{000.1}&&\textbf{0.167}& \textbf{0.103}\\\hline
		IIHT&20&10&0.325&0.228 &000.5&&0.315&0.233\\
		&20&20&0.323&0.197 &000.6&&0.301&0.198\\
		&40&20&0.319& 0.159&000.6&&0.305&0.227\\
		&40&40&0.318&0.162&000.7&&0.312&0.173\\\hline
		GraSP&20&10&0.324& 0.265&000.7&&0.336&0.302\\
		&20&20&0.312&0.263&000.7&&0.328&0.273\\
		&40&20&0.286&0.237&000.8&&0.313&0.262\\
		&40&40&0.294&0.258&000.9&&0.327&0.235\\\hline
		SP&20&10&0.337& 0.169&000.4&&0.344&0.183\\
		&20&20&0.335&0.158&000.4&&0.342&0.129\\
		&40&20&0.338&0.146&000.6&&0.353&0.187\\
		&40&40&0.334&0.138&000.9&&0.355&0.159\\\hline
	    LNA&20&10&0.266 &0.282  &000.3 &&0.378 &0.237 \\
		&20&20& 0.275&0.269 &000.3 &&0.346 & 0.245\\
		&40&20& 0.254&0.257 &000.4 && 0.361&0.235 \\
		&40&40& 0.253& 0.263&000.5 &&0.334 &0.239 \\
		\cline{1-9}
		\multicolumn{9}{c}{ DLBCL}\\\hline
		SCoRe& & & 0.533& 0.315&083.4 && 0.546 &0.307 \\\hline
		GPNA&50&50&\textbf{0.267} & \textbf{0.166}&\textbf{000.6}&&\textbf{0.313}&\textbf{0.213} \\
		&50&100&\textbf{0.243}& \textbf{0.167}&\textbf{000.6}&&\textbf{0.339} &\textbf{0.225}\\
		&100&100&\textbf{0.234}&\textbf{0.159} &\textbf{000.6}&&\textbf{0.324} &\textbf{0.212} \\
		&100&150&\textbf{0.233}&\textbf{0.158} &\textbf{000.7} &&\textbf{0.311} &\textbf{0.215} \\\hline
		IIHT&50&50&0.417&0.352 &039.2&&0.445&0.326\\
		&50&100&0.412&0.346 &042.7&&0.437&0.317\\
		&100&100&0.403&0.337 &045.6&&0.431&0.314\\
		&100&150&0.408&0.346 &046.7&&0.438&0.324\\\hline
		GraSP&50&50&0.456&0.434 &235.5&&0.441&0.362\\
		&50&100&0.458& 0.457&254.8&&0.451&0.353\\
		&100&100&0.432&0.442 &228.6&&0.439&0.351\\
		&100&150&0.422& 0.446&232.7&&0.440&0.363\\\hline
		SP&50&50&0.426&0.423 &013.5&&0.435&0.334\\
		&50&100&0.428& 0.437&014.8&&0.443&0.341\\
		&100&100&0.416&0.425 &015.6&&0.432&0.331\\
		&100&150&0.419& 0.432&017.4&&0.426&0.337\\\hline
		 LNA&50&50&0.398 & 0.451 &000.8 &&0.425 &0.366 \\
		&50&100&0.414 & 0.417&000.9 && 0.407&0.351 \\
		&100&100&0.386 &0.422 &001.1 && 0.396& 0.348\\
		&100&150& 0.381&0.436 &001.2 && 0.413&0.359 \\
		\hline		
	\end{tabular}
\end{table}

{\bf (f) Comparison for Example \ref{real2}.}  Finally, we report results of five methods for analysing two real datasets: Mouse data and DLBCL.   For mouse gene expression data, similar to \cite{Luoc}, we choose $p_{1} =163$ single nucleotide polymorphisms (SNPs corresponding to $X$)   and $p_{2} =215$ genes (corresponding to $Z$)   of $ n =294$ for  analysis. Again random splitting procedure is employed. At each split, 140 observations are randomly chosen as the testing data and the remaining 154 observations are the training data. The random splitting is repeated 100 times. We choose different sparsity and the average results are reported in Table \ref{Table5} and   display the better behaviour of GPNA since it runs much faster and obtains lower MSE (meaning better predictions), smaller CCV. For DLBCL, results present in Table \ref{Table5}, where the random splitting procedure being same as Example \ref{real1}. Similarly, GPNA obtains lower MSE (meaning better predictions), smaller CCV and runs the fastest, such as 0.6 seconds consumed when $s_1=s_2=50$ v.s. 83.4, 39.2, 235.5, 13.5 and 0.9 seconds  by the other five methods, which demonstrate better performance of GPNA.

\section{Conclusions and Future work} \label{sec:con}

The SCL model proposed in this paper not only fulfils the tasks of classification or regression for each dataset  but also explores the relationship between two datasets. The usage of the double sparsity constraints makes it more efficient for feature selections. To solve the SCL problem, the optimality conditions have been investigated, leading to a gradient projection strategy in the algorithm. To accelerate the convergence, we employed a Newton step when the iteration met some conditions. The final developed gradient projection Newton algorithm has proven to be global and at least quadratic convergent and possessed an excellent numerical performance.  We feel that the proposed method is capable of addressing some other general sparsity constrained optimization problems.

As pointed out by our referee,  it is an interesting topic to apply the developed techniques and method into dealing with the multi-model problems in particular for some practical applications, such as regional climate prediction. We leave this as future research.

\section*{Acknowledgments}

The authors would like to thank the Principal Editor and
the anonymous referee for their helpful suggestions.
\bibliographystyle{abbrv}
\bibliography{references}
\end{document}